\documentclass[12pt,reqno]{amsart}
\usepackage{hyperref}
\usepackage{cleveref}
\usepackage{amsmath,amssymb}
\usepackage[normalem]{ulem}
\usepackage{caption}
\usepackage{graphicx}
\usepackage{float}
\usepackage{pgf,tikz}
\usepackage{amsmath,amscd}
\usepackage{tikz-cd}
\usepackage{mathtools}
\usepackage{mathrsfs}

\DeclarePairedDelimiter\ceil{\lceil}{\rceil}
\DeclarePairedDelimiter\floor{\lfloor}{\rfloor}

\newcommand \supp{\operatorname{supp}}

\newcommand \ass{\operatorname{Ass}}

\newcommand \HH{\mathcal{H}}

\newcommand \kk{\mathbb{K}}

\newcommand \NN{\mathbb{N}}
\newcommand \m{\mathfrak{m}}
\def\v{{\bf v}}
\def\R{{\mathcal R}}

\newtheorem{theorem}{Theorem}[section]

\newtheorem{lemma}[theorem]{Lemma}
\newtheorem{proposition}[theorem]{Proposition}
\newtheorem{example}[theorem]{Example}
\newtheorem{obs}[theorem]{Observation}
\newtheorem{question}[theorem]{Question}
\newtheorem{remark}[theorem]{Remark}
\newtheorem{corollary}[theorem]{Corollary}

\newtheorem{notation}[theorem]{Notation}

\begin{document}
\title[On the resurgence and asymptotic resurgence of homogeneous ideals]{On the resurgence and asymptotic resurgence of homogeneous ideals}
\author{A. V. Jayanthan}
\email{jayanav@iitm.ac.in}
\address{Department of Mathematics,
	Indian Institute of Technology Madras, Chennai, INDIA - 600036}
\author{Arvind Kumar}
\email{arvindkumar@cmi.ac.in}
\address{Department of Mathematics, Chennai Mathematical Institute, Siruseri
		Kelambakkam, Chennai, India - 603103}
\author{Vivek Mukundan}
\email{vmukunda@iitd.ac.in}
\address{Department of Mathematics, 
Indian Institute of Technology Delhi, New Delhi, India - 110016}
\maketitle
	
\begin{abstract}
Let $\kk$ be a field and $R = \kk[x_1, \ldots, x_n]$. We obtain an improved upper bound for asymptotic resurgence of squarefree monomial ideals in $R$. We study the effect on the resurgence when sum, product and intersection of ideals are taken. 
We obtain sharp upper and lower bounds for the resurgence and asymptotic resurgence of cover ideals of finite simple graphs in terms of associated combinatorial invariants. We also explicitly compute the resurgence and asymptotic resurgence of cover ideals of several classes of graphs. We characterize a graph being bipartite in terms of the resurgence and asymptotic resurgence of edge and cover ideals. We also compute explicitly the resurgence and asymptotic resurgence of edge ideals of some classes of graphs.
\end{abstract}
	
\section{Introduction}
The main objective of this article is to study the containment  between the ordinary and symbolic powers of ideals in a polynomial ring $R$. While there is a nice geometric description for the symbolic powers of ideals, there is no such description for ordinary powers. For example, if $X$ is a smooth variety over a perfect field and $P$ denotes the prime ideal of $X$, then $P^{(n)} = \displaystyle{\bigcap_{\mathfrak{m}} \mathfrak{m}^n}$, where the intersection is taken over all closed points $\mathfrak{m} \in V(P)$, by Zariski-Nagata theorem \cite{Zar49, Nag62}. For an ideal $I$ in a Noetherian ring $R$, the $s$-th symbolic power of $I$ is  defined as $\displaystyle I^{(s)}=\bigcap_{P\in\ass(I)}(I^sR_P\cap R)$.  By definition, it is clear that $I^s\subseteq I^{(s)}$ for all $s \in \mathbb{N}$, but equality of these two objects are rarely satisfied. The \textit{containment problem} concerns finding the smallest  $s \in \mathbb{N}$ for a given $t \in \mathbb{N}$ such that the $s$-th symbolic power $I^{(s)}$ is contained in the $t$-th ordinary power $I^t$. This arose as a study of the consequences of the comparison of the ordinary and symbolic power topologies (see \cite{schenzel86,swanson00}). There have been a lot of study to understand this equivalence. It was proved that if $I$ is a radical ideal of \textit{big height} $h$ in a regular ring, then $I^{(ht)} \subset I^t$ for all $t \in \mathbb{N}$, \cite{swanson00, ELS01, HH02, MS17}, where the big height is defined to be the maximum height among all the associated primes of $I$. Thus, $P^{(4)}\subseteq P^2$ if $P$ is a height two prime ideal in a three dimensional regular local ring. Huneke questioned whether this can be sharpened to satisfy $P^{(3)}\subseteq P^2$. It is proved affirmatively in the case of space monomial curves \cite{G20}. More general classes satisfying this question can be found in \cite[Theorem 5.1]{GHM20}. Subsequently, Harbourne conjectured in  \cite[Conjecture 8.4.3]{seshariconstantcounterexample}  that if $I$ is a radical ideal in a regular ring $R$, then $I^{(ht)}\subseteq I^t$ could be sharpened to $I^{(ht-h+1)}\subseteq I^t$ for all $t \in \mathbb{N}$, where $h$ is the  big height of $I$. Though there are large classes of ideals which satisfy Harbourne's conjecture \cite{seshariconstantcounterexample,GH19}, counter examples have been constructed in \cite{DST13} (see \cite{harbournesecelenau15} for more information on these counterexamples). Grifo in \cite{G20} questioned whether Harbourne's conjecture is true asymptotically i.e., $I^{(ht-h+1)}\subseteq I^t$ for all $t\gg 0$. This is known as the \textit{stable Harbourne conjecture}. There are no counterexamples to this conjecture yet. For some recent advances in this direction, see \cite{GHM20,GHM21,DD20}. 

In order to better study the containment problem, Bocci and Harbourne in \cite{CB10} introduced an invariant, the \textit{resurgence }of $I$. The resurgence of an ideal $I$ is defined as
\begin{align*}
	\rho(I)=\sup\left\{ \dfrac{s}{t} \; : \; s,t \in \mathbb{N} \text{ and } I^{(s)} \not\subset I^t\right\}.
\end{align*}
Of course, if $\displaystyle\frac{s}{t}>\rho(I)$, then $I^{(s)}\subseteq I^t$. As a consequence of the results in \cite{ELS01,HH02,MS17}, a natural upper bound for the resurgence  of a radical ideal in a regular ring is the big height $h$. Grifo in \cite{G20}  showed that the stable Harbourne conjecture is true when this upper bound is not achieved i.e., when  $\rho(I)<h$. Such ideals are referred to as ideals having \textit{expected resurgence}. 
Necessary conditions for ideals having expected resurgence have been explored in \cite{GHM20,GHM21}.
Guardo, Harbourne and Van Tuyl in \cite{EBA} introduced a refinement to the resurgence called the \textit{asymptotic resurgence}. The asymptotic resurgence is defined as 
\begin{align*}
	\rho_a(I)=\sup\left\{ \frac{s}{t}\; : \; s,t \in \mathbb{N} \text{ and } I^{(sr)} \not\subset I^{tr}\text{ for  }r\gg0\right\}.
\end{align*}
They went on to show that if $I$ is a homogeneous ideal in finitely generated graded $\kk$-algebra $R$, then $\displaystyle 1\leq \frac{\alpha(I)}{\hat{\alpha}(I)}\leq\rho_a(I)\leq \rho(I)$ where  $\alpha(I)$ denote the minimal degree of an element in $I$ and $\hat{\alpha}(I)$ denotes the {\it Waldschmidt constant} defined as $\displaystyle \hat{\alpha}(I) = \underset{s \to \infty}{\lim}\frac{\alpha(I^{(s)})}{s}$. It follows from  \cite{G20} and \cite{GHM20} that the stable Harbourne conjecture is true when the resurgence or the asymptotic resurgence is strictly smaller than the big height $h$. 
A great source of examples for which the resurgence and asymptotic resurgence  are known, comes either from the geometric side (\cite{CB10,BHproceedings10,DHNSST15,HKZ20,harbournesecelenau15}) or from combinatorial side (see \cite{LM15, JKV19, GHOS, JK19}).


%

In general, the computation of resurgence and asymptotic resurgence is a tough task, even for well structured classes of homogeneous ideals in polynomial rings. A more approachable method has been constructed for asymptotic resurgence by DiPasquale, Francisco, Mermin and Schweig in \cite{MCM19}. In \cite{DD20}, DiPasquale and Drabkin studied the resurgence via asymptotic resurgence. They proved that if $\rho_a(I) < \rho(I)$, then one can reduce the computation of the resurgence to a finite process. As a consequence, they proved that if the symbolic Rees algebra is Noetherian, then the resurgence is a rational number. If the symbolic Rees algebra is generated by linear and degree $n$ forms, then we generalize \cite[Theorem 6.2]{GHM20} to obtain a general upper bound for resurgence (\Cref{gen-GHM}). We also study the effect on the resurgence and asymptotic resurgence when  product, sum, intersection of ideals are taken. We show that if $I$ and $J$ are ideals in different set of variables, then the resurgence and asymptotic product can be computed from those invariants of the individual ideals, (\Cref{dis-res}). We also show that if the resurgence of a finite collection of ideals in distinct set of variables are equal to $1$, then the resurgence of their sum can be computed by knowing the least integer for which their ordinary and symbolic powers are not equal, (\Cref{res-sum1}, \Cref{sum-resurgence}).

It is known that Harbourne's conjecture is true for squarefree monomial ideals, \cite{CEHH17}. We show that an improved containment exists, albeit asymptotically. As a consequence, we improve an upper bound, given in \cite{DD20}, for the asymptotic resurgence. We write these results in terms of cover ideals of hypergraphs. This is because any squarefree monomial ideal can be seen as a cover ideal of a hypergraph. Once we are in the hypergraph theory, we have combinatorial tools to assist us. We refer to Section 3 for the definition of hypergraph and \cite{HV08} for a detailed study on the cover ideals of hypergraphs. We also relate the resurgence and asymptotic resurgence of cover ideals of a hypergraph and its subhypergraphs. 
\begin{theorem}[\Cref{general-upperbound} and \Cref{induced-asym-resurgence}]\label{1.1}
	Let $\mathcal{H}$ be a hypergraph, $J(\HH)$ be its cover ideal, $\chi(\mathcal{H})$ be the chromatic number of $\HH$ 
	and $h$ be the big height of $J(\HH)$.  Then 
	\begin{enumerate}
		\item $\displaystyle J(\mathcal{H})^{(rh-h)} \subset J(\mathcal{H})^r$ for all $\displaystyle r \geq \chi(\mathcal{H}).$ 
		\item $\displaystyle \rho_a(J(\mathcal{H})) \leq h -\frac{1}{\chi(\mathcal{H})}.$
		\item  $\displaystyle \rho(J(\HH')) \leq \rho(J(\HH))$ and  $\displaystyle \rho_a(J(\HH')) \leq \rho_a(J(\HH))$ for any subhypergraph $\mathcal{H}'$  of $\HH$.
	\end{enumerate}
\end{theorem}

Much of our work attempt to compute the resurgence and asymptotic resurgence for various classes of combinatorially enriched ideals  such as cover and edge ideals of finite simple graphs (see Section 2 for the definition of cover ideals and edge ideals). We first prove a variant of the stable Harbourne conjecture for cover ideals of finite simple graphs. Since the cover ideals are defined using combinatorial data, it is natural to expect that the algebraic invariants associated with them are related to the combinatorial invarinats associated with the corresponding graph. In the case of resurgence and asymptotic resurgence, we  obtain sharp lower and upper bounds in terms of combinatorial invariants such as clique number, $\omega(G)$, independence number $\alpha(G)$ and chromatic number $\chi(G)$ (see Section 2 for their definitions).
\begin{theorem}[\Cref{res-upper}, \Cref{res-upper1}] \label{1.2}
	Let $G$ be a connected graph on $n$ vertices. Then
	\begin{enumerate}
		\item $\displaystyle J(G)^{(2r-2c)} \subset J(G)^r$ for every $\displaystyle r \geq c\chi(G).$
		\item $\displaystyle J(G)^{(2r-2c-1)} \subset J(G)^r$ for every $\displaystyle r \geq c\chi(G)+1.$
		\item $\displaystyle \max\left \{ 2-\frac{2}{\omega(G)}, 2 - \frac{2 \alpha(G)}{n} \right\} \leq \rho_a(J(G)) \leq \rho(J(G)) \leq 2-\frac{2}{\chi(G)}.$
	\end{enumerate}  
\end{theorem}
As an immediate consequence, we obtain the resurgence and asymptotic resurgence  of cover ideals of perfect graphs (for example, bipartite graphs, chordal graphs, complete multipartite graphs, even-wheel graphs etc.).
The above theorem also gives values to integer $N$, for a given $C$, of Question 2.2 raised by Grifo in \cite{G20}. While in \Cref{1.1} we provide an upper bound for asymptotic resurgence for an arbitrary squarefree monomial ideal, \Cref{1.2} improves it when the ideal is a cover ideal of a finite simple graph. Moreover, we show that the improved upper bound works as an upper bound for the resurgence as well.


We then proceed to compute the resurgence and asymptotic resurgence of cover ideals of specific classes of graphs which are not perfect graphs. We first compute the resurgence and asymptotic resurgence of cover ideals of odd cycles, (\Cref{odd-cycle}). If a graph $G$ is bipartite, then it is known that $I^{(n)} = I^n$ for all $n \geq 1$, where $I = I(G)$ or $I = J(G)$, \cite{SVV, HHTV}. Hence $\rho(I) = 1$ in this case. It may be noted that, in general, for an ideal $I$, $\rho(I) = 1$ need not necessarily imply that $I^{(n)} = I^n$ for all $n \geq 1$. If $I$ is either the cover ideal or the edge ideal of a graph, then we prove that
$G$ is bipartite if and only if $\rho_a(I) = 1$ if and only if $\rho(I) = 1$, (\Cref{rho=1}, \Cref{rhoI=1}).

Given that the computation of resurgence and asymptotic resurgence is a heavy task, it is natural to look for methods to reduce the difficulty level. One way to make the job easier is by reducing the computation to its induced subgraphs. In the case of cover ideals, we show that we can achieve this when a graph is clique-sum of its induced subgraphs (see Section 2 for the definition of clique-sum). 
\begin{theorem}[\Cref{clique-sum}]
Let $\displaystyle G=G_1 \cup G_2$ be  a clique-sum of $G_1$ and $G_2$. Then :
\begin{enumerate}
	\item For any $t \geq 1$, $\displaystyle J(G)^t=J(G_1)^t \cap J(G_2)^t$.
	\item For any $s \geq 1$, $\displaystyle J(G)^{(s)}=J(G_1)^{(s)} \cap J(G_2)^{(s)}$.
	\item $\displaystyle \rho(J(G))=\max\{\rho(J(G_1)),\rho(J(G_2))\}$.
	\item $\displaystyle \rho_a(J(G))=\max\{\rho_a(J(G_1)),\rho_a(J(G_2))\}$.
\end{enumerate}  
\end{theorem}
As a consequence of this result, we show that for a non-bipartite Cactus graph, the resurgence is equal to the resurgence of the smallest odd cycle present in the graph, (\Cref{cactus}). 

For the class of edge ideals, the resurgence and asymptotic resurgence are known only for a handful of classes.
In \cite{MCM19}, an explicit formula for the asymptotic resurgence of edge ideals was given in terms of fractional chromatic number. The resurgence of edge ideals of odd cycles was first computed in \cite{JKV19}. This was generalized to non-bipartite unicyclic graphs in \cite{GHOS} and generalized to the case of graphs containing one odd cycle in \cite{JK19}. In the final section of our article, we relate the resurgence and asymptotic resurgence of edge ideals of a graph and its induced  subgraphs, (\Cref{tech2}) and we explicitly compute the resurgence and asymptotic resurgence of edge ideals of a certain class of graphs. Let $G$ be a clique-sum of bipartite graphs and odd cycles. First, we give a decomposition for $I(G)^{(s)}$ in terms of ordinary powers (\Cref{gen}). We conclude our article by computing the resurgence and asymptotic resurgence in terms of size of odd cycles present in $G$:
\begin{theorem} [\Cref{res-cactus-edge}, \Cref{res-edgeideal}] Let $G$ be a clique-sum of odd cycles and bipartite graphs and $I(G)$ denotes its edge ideals. Let $2n+1$ be the smallest size of odd cycles in $G$. 
\begin{enumerate}
    \item Then $\displaystyle \rho_a(I(G)) = \frac{2n+2}{2n+1}$.
    \item If all odd cycles in $G$ are of same size, then $\displaystyle\rho(I(G)) = \frac{kn+k}{kn+1}$, where $k \geq 2$ is the maximum number of odd cycles which are pairwise at a distance two or higher.
\end{enumerate}
\end{theorem}
In \cite{MCM19}, DiPasquale et al. asked if the resurgence and asymptotic resurgence of edge ideals of graphs are equal. The first counter example was given by Andrew Conner, \cite[Example 4.4]{DD20}. The  above theorem gives a class of examples for which the asymptotic resurgence is strictly less than the resurgence.

The article is organized as follows: We collect the notation and the preliminary concepts in Section 2. In the next section, we prove the results on resurgence and asymptotic resurgence of homogeneous ideals in polynomial rings. The computation of resurgence and asymptotic resurgence of cover ideals of finite simple graphs are done in Section 4. In Section 5, we deal with the resurgence of edge ideals.

\section{Preliminaries}

In this section, we collect notation and terminology used  in the  subsequent sections. We begin with recalling the combinatorial preliminaries.

Let $G$ be a finite simple  graph with the vertex set $V(G)$ and edge set $E(G)$. For $A \subset V(G)$, $G[A]$ denotes the \textbf{induced	subgraph} of $G$ on the vertex set $A$, i.e., for $x, y \in A$, $\{x,y\} \in E(G[A])$ if and only if $ \{x,y\} \in E(G)$. For a vertex $x$, $G \setminus x$ denotes the  induced subgraph of $G$
on the vertex set $V(G) \setminus \{x\}$. A vertex $x \in V(G)$ is said to be a \textbf{cut vertex} if $G \setminus x$ has  more connected components than $G$. A subset $\Gamma\subset V(G)$ is said to be a \textbf{vertex cover} if for any edge $e \in E(G)$, $\Gamma \cap e \not= \emptyset$. A vertex cover is said to be \textbf{minimal vertex cover} if no proper subset of it is a vertex cover. 
A \textbf{complete graph} on $n$ vertices $\{x_1, \ldots, x_n\}$, denoted by $K_n$, is a graph with the edge set $\{\{x_i, x_j\} ~:~ 1\leq i< j \leq n\}$. 
A subset $U$ of $V(G)$ is said to be a \textbf{clique} if $G[U]$ is a complete graph. The \textbf{clique number}
of a graph ${G}$, denoted by $\omega({G})$, is the maximum size of maximal cliques of ${G}$. 
The \textbf{chromatic number} of a graph $G$ is the minimum number of colors required to color vertices of $G$ so that adjacent vertices have different color. The \textbf{chromatic number} of $G$ is denoted by $\chi(G)$.

A \textbf{cycle} in $G$ is a sequence of distinct vertices $x_1, \dots, x_n$ such that $\{x_i,x_{i+1}\}$ is an edge for all $i = 1, \dots, n$ (here $x_{n+1} \equiv x_1$). A cycle on $n$ distinct vertices is called an \textbf{$n$-cycle} and often denoted by $C_n.$ An $n$-cycle is said to be an {\bf even cycle} if $n$ is even, and is said to be an \textbf{odd cycle} if $n$ is odd. A \textbf{block} of a graph is a maximal induced subgraph without a cut vertex. 
A graph $G$ is said to be a \textbf{cactus graph} if  each block of $G$ is either a cycle or an edge. A graph $G$ is said to be \textbf{bipartite} if we can write $V(G) = A \sqcup B$ such that $E(G[A]) = \emptyset$, $E(G[B]) = \emptyset$. A graph $G$ is said to be \textbf{chordal} if the maximum size of an induced cycle in $G$ is $3$.
A graph $G$ is said to be \textbf{perfect} if $\chi(G[A])=\omega(G[A])$ for all $A \subset V(G).$ 

Let $G_1$ and $G_2$ be graphs. If $G_1 \cap G_2 = K_r$ with $G_1 \not= K_r$ and $G_2 \not= K_r$, then $G_1 \cup G_2$ is called the \textbf{clique-sum} of $G_1$ and $G_2$ along $K_r$. 
The \textbf{join} of $G_1$ and $G_2$, denoted by $G_1*G_2$ is the graph with vertex set $V(G_1) \sqcup V(G_2)$ and edge set $E(G_1) \sqcup E(G_2) \sqcup \{\{x,y\} : x \in V(G_1), y \in V(G_2)\}$. A graph $G$ is said to be a \textbf{complete multipartite graph} if we can write $V(G) = V_1 \sqcup \cdots \sqcup V_k$ with the edge set $E(G) = \{\{x_i, x_j\} : x_i \in V_r \text{ and } x_j \in V_s \text{ for } 1\leq r < s \leq k\}$. If $|V_i| = n_i$, then $G$ is usually denoted by $K_{n_1,\ldots,n_k}$. For a graph $G$, the \textbf{complement} of $G$, denoted by $G^c$, is the graph with vertex set $V(G)$ and edge set $E(G^c) = \{\{x_i, x_j\} : \{x_i, x_j\} \notin E(G)\}.$

A \textbf{hypergraph} $\HH$ is a pair $(V(\HH), E(\HH))$, where $V(\HH)$ is a set of elements called the vertices and $E(\HH)$ is a non-empty set of subsets of $V(\HH)$. We further assume that $V(\HH) < \infty$,  $|e| \geq 2$ for each $e \in E(\HH)$, and for $e_1, e_2 \in E(\HH)$, $e_1 \not\subset e_2$.  A \textbf{vertex cover} $\Gamma$ of $\HH$ is a subset of $V(\HH)$ that satisfies $\Gamma \cap e \neq \emptyset$ for each $e \in E(\HH)$. Let $\chi(\HH)$ denote the \textbf{chromatic number} of a hypergraph $\HH$, defined to be the least number of colors required to color the vertices so that not all vertices of each edge are of the same color.

Throughout this paper, all the graphs that we consider have finite vertices, no isolated vertices, no loops and no multiple edges.
For any undefined terminology and basic properties of graphs, we refer the reader to \cite{west}.

Let $G$ be a graph over the vertex set $V(G) = \{x_1, \dots, x_n\}$. The \textbf{edge ideal} of $G$, denoted by $I(G)$, is the ideal generated by  $\{ xy ~:~ \{x,y\} \in E(G)\} \subset R=\kk[x_1,\cdots,x_n].$ The ideal generated by $\{x_{i_1}\cdots x_{i_r} ~:~ \{x_{i_1}, \ldots, x_{i_r}\} \text{ is a vertex cover of } G\} \subset R$ is called the \textbf{cover ideal} of $G$, denoted by $J(G)$. It is known that for any graph $G$, the cover ideal $J(G)$ is the Alexandar dual of the edge ideal $I(G)$, \cite{HV08}, i.e., $I(G) = \cap_{C \in \mathcal{C}(G)} P_C$ and $J(G) = \cap_{\{x_i, x_j\} \in E(G)} (x_i, x_j)$, where $\mathcal{C}(G)$ denotes the collection of all minimal vertex covers of  $G$  and for $C \in \mathcal{C}(G)$, $P_C$ denote the monomial prime ideal generated by the elements of $C$.

Let $\HH$ be a hypergraph with vertex set $V(\HH)= \{x_1, \ldots, x_n\}$. The \textbf{cover ideal of the hypergraph} $\HH$, denoted by $J(\HH)$,  is the ideal in $\kk[x_1, \ldots, x_n]$ generated by the set 
$$\displaystyle \left\{\prod_{x \in \Gamma} x ~ : ~ \Gamma \text{ is a vertex cover of }  \HH\right\}.$$ For $e \in E(\HH)$, let $P_e = ( x ~:~ x \in e ) \subset \kk[x_1, \ldots, x_n]$. Then, it is easy to see that $\displaystyle J(\HH) = \bigcap_{e \in E(\HH)} P_e$. For more properties of cover ideal of hypergraph, we refer the reader to \cite{HV08}. 

Let $S$ be a Noetherian ring, and  let $I \subset S$ be an ideal. The \textbf{Rees algebra}, denoted by $\R(I)$, and the \textbf{symbolic Rees algebra}, denoted by $\R_s(I)$, of $I$ are defined to be
	$$\R(I) := \bigoplus_{n \ge 0} I^nt^n \subset S[t] \text{ and } \R_s(I) := \bigoplus_{n \ge 0} I^{(n)}t^n \subset S[t].$$

\section{resurgence of a homogeneous ideal}
In this section, we study the relationship between $\rho(I\star J)$ and $\rho(I),\rho(J)$ for different operations $\star$ between the ideals $I, J$. While these results are useful on their own, they also provide us the necessary tools required to prove results in the later sections.

In \cite[Theorem 6.2]{GHM20},  the authors obtained an upper bound for the resurgence of some classes of height two prime ideals in three dimensional regular local rings. In the following, we get a generalization of their result.

\begin{theorem}\label{gen-GHM}
Let $I$ be a nonzero proper ideal in a Noetherian ring $S$. If $\displaystyle\R_s(I)=S[It,I^{(n)}t^n]$ for some $n \geq 2$ and there exists an ideal $P$ such that $PI^{(n)} \subset I^n $ and $I^{(n)} \subset P^kI^{n-1}$ for some $k \geq 1$, then $I^{(nkq+nq)} \subset I^{nkq+nq-q}$ for all $q \in \mathbb{N}$, and $\displaystyle \rho(I) \leq \frac{nk+n}{nk+n-1}$.
\end{theorem}

\begin{proof}
Since $\displaystyle\R_s(I)=S[It,I^{(n)}t^n]$, we have, for any $a\geq 1$,  $\displaystyle I^{(a)}=(I^{(n)})^q I^r$,  where $a =nq+r$ with $q \geq 0$ and $0 \leq r \leq n-1$.
Now, let $s,t \in \mathbb{N}$ be such that $I^{(s)} \not\subset I^t$. Let $q,r \in \mathbb{N}\cup \{ 0\}$ be such that $s=(nk+n)q+r$ with $0 \leq r \leq nk+n-1$. Then,
\begin{align}\label{eqn2}
I^{(s)}&=I^{((nk+n)q+r)} =\Big(I^{(n)}\Big)^{qk+q} I^{(r)} =\Big(I^{(n)}\Big)^q\Big(I^{(n)}\Big)^{qk} I^{(r)} \nonumber \\& \subset \Big(P^kI^{n-1}\Big)^q\Big(I^{(n)}\Big)^{qk} I^{(r)} =  I^{nq-q}\Big(PI^{(n)}\Big)^{qk} I^{(r)} \nonumber \\& \subset  I^{nq-q}\Big(I^{n}\Big)^{qk} I^{(r)} =       I^{nkq+nq-q} I^{(r)}. 
\end{align}
If $0 \leq r \leq n-1$, then $I^{(r)}=I^r$ as $\R_s(I)=S[It,I^{(n)}t^n]$. Consequently, by \eqref{eqn2}, $t \geq s-q+1$, and hence,   
$\displaystyle\frac{s}{t} \leq \frac{nk+n}{nk+n-1}$. 
Assume that  $n \leq r \leq nk+n-1$. Write $r=nq_1+r_1$ for some $q_1 \in \NN$ and $0 \leq r_1 \leq n-1$. Since $1 \leq q_1 \leq k$, we have $\displaystyle I^{(n)} \subset P^kI^{n-1} \subset P^{q_1-1}I^{n-1}$. Thus, we get
\begin{align}\label{eqn3}
I^{(r)} & = \Big(I^{(n)}\Big)^{q_1}I^{r_1}=I^{(n)} \Big(I^{(n)}\Big)^{q_1-1}I^{r_1}\nonumber\\ 
& \subset  P^{q_1-1}I^{n-1}\Big(I^{(n)}\Big)^{q_1-1}I^{r_1}= \Big(PI^{(n)}\Big)^{q_1-1}I^{r_1+n-1} \nonumber \\ 
& \subset \Big(I^n\Big)^{q_1-1}I^{r_1+n-1}=I^{r-1}.
\end{align}
Therefore, it follows from \eqref{eqn2} and \eqref{eqn3} that $t \geq s-q$ so that $\displaystyle\frac{s}{t} \leq \frac{(nk+n)q+r}{(nk+n-1)q+r} \leq \frac{nk+n}{nk+n-1}$. Hence,  $\displaystyle \rho(I) \leq \frac{nk+n}{nk+n-1}$.
\end{proof}
The bound in \Cref{gen-GHM} is a tight bound. We illustrate this in the following example.
\begin{example}{\em
For $n \geq 2$, take $I=I(C_{2n-1}) \subset S = \kk[x_1, \ldots, x_{2n-1}]$. Then, it follows from the proof of \cite[Theorem 3.4]{GHOS} that $\R_s(I) = S[It, I^{(n)}t^n]$ and $ I^{(n)}=I^n+(x_1\cdots x_{2n-1})$. Now,  $\mathfrak{m}I^{(n)} =\mathfrak{m}I^n + x_1\cdots x_{2n-1}\mathfrak{m} \subset I^n$, where $\mathfrak{m}$ is the unique homogeneous maximal ideal of $S$. Also, $I^{(n)} =I^n+(x_1\cdots x_{2n-1}) \subset \mathfrak{m}I^{n-1}.$ Thus, by Theorem \ref{gen-GHM}, $\displaystyle \rho(I) \leq \frac{2n}{2n-1}.$ In fact $\displaystyle \rho(I) = \rho_a(I) = \frac{ 2n}{2n-1}$, by \cite[Theorem 5.11]{JKV19}.
}
\end{example}

We now study the resurgence and asymptotic resurgence of intersection and product of nonzero proper ideals in a  Noetherian ring.
\begin{proposition}\label{int-res}
Let $I, J $ be nonzero proper ideals in a Noetherian ring $S$. If for all $s, t \in \mathbb{N}$, $ (I \cap J)^t = I^t \cap J^t$ and $(I \cap J)^{(s)} = I^{(s)} \cap J^{(s)}$.  Then : \begin{enumerate}
	\item $\rho(I\cap J) \leq \max\{ \rho(I),\rho(J)\}$.
	\item $\rho_a(I\cap J) \leq \max\{ \rho_a(I),\rho_a(J)\}$.
\end{enumerate}
\end{proposition}
\begin{proof} (1) Let $s,t \in \mathbb{N}$ be such that $\displaystyle \frac{s}{t} > \max\{\rho(I), \rho(J)\}$. Therefore, $I^{(s)} \subset I^t$ and $J^{(s)} \subset J^t$. Now, $(I\cap J)^{(s)} =I^{(s)} \cap J^{(s)} \subset I^t \cap J^t =(I \cap J)^t$. Thus, $\max\{\rho(I), \rho(J)\}$ is an upper bound for the set $\displaystyle \Big\{\frac{s}{t} : s,t \in \mathbb{N}  \text{ and } (I\cap J)^{(s)} \not\subset (I\cap J)^t  \Big\}$, and hence, $\rho(I\cap J) \leq \max\{ \rho(I),\rho(J)\}$. 
\par 
(2) By \cite[Proposition 4.2]{MCM19} and \cite[Lemma 2.2]{SHJ21},  $\displaystyle\rho_a(I)=\limsup_{m \to \infty}\rho(I,m),$ where $\displaystyle\rho(I,m)= \sup \Big\{\frac{s}{t} : s,t \geq m \text{ and } I^{(s)} \not\subset I^{t}\Big\}$. For $m \in \NN$, let $\displaystyle A_m(I) := \left\{\frac{s}{t}  :  s, t \geq m \text{ and } I^{(s)} \not\subset I^t \right\}$. Then, $A_{m+1}(I) \subset A_m(I)$ and hence $\displaystyle\limsup_{m \to \infty}\rho(I,m)=\lim_{m \to \infty}\rho(I,m)$. Thus, we need to  prove that $$\displaystyle\lim_{m\to \infty} \rho(I\cap J, m) \leq \max\Big\{\lim_{m \to \infty}\rho(I,m), \lim_{m\to \infty} \rho(J,m)\Big\}.$$ For this, it is enough to prove that $\rho(I\cap J,m) \leq \max\{\rho(I,m),\rho(J,m)\}$ for every $m \in \mathbb{N}$. Let $s,t \geq m$ be integers such that $\displaystyle \frac{s}{t} > \max\{\rho(I,m), \rho(J,m)\}$. Therefore, $I^{(s)} \subset I^{t}$ and $J^{(s)} \subset J^{t}$. Now, $(I\cap J)^{(s)} =I^{(s)} \cap J^{(s)} \subset I^{t} \cap J^{t} =(I \cap J)^{t} $. 
Hence, $\rho(I\cap J,m) \leq \max\{ \rho(I,m),\rho(J,m)\}$ for all $m \geq 1$. Consequently, $$\lim_{m \to \infty} \rho(I\cap J,m) \leq \lim_{m \to \infty}\max\{ \rho(I,m),\rho(J,m)\}=\max\Big\{\lim_{m \to \infty}\rho(I,m), \lim_{m\to \infty} \rho(J,m)\Big\}.$$ Hence, the assertion follows. 
\end{proof}

\begin{remark}\label{remark-sympow}
{\em
If $I_1$ and $I_2$ are radical ideals in a Noetherian ring $S$ and $I = I_1 \cap I_2$, then $I^{(s)} = I_1^{(s)} \cap I_2^{(s)}$. Therefore, the above result is beneficial in studying the resurgence when $I_1$ and $I_2$ are radical ideals (for example, squarefree monomial ideals in a polynomial ring) and $I = I_1 \cap I_2$. We will see an immediate application in the following result. More applications will be given in the next section.
}
\end{remark}

\begin{proposition}\label{dis-res}
Let $I \subset \mathbb{K}[x_1,\ldots,x_m]$ and $J \subset \mathbb{K}[y_1,\ldots,y_n]$ be nonzero proper ideals. Then, for $IJ \subset \mathbb{K}[x_1,\ldots,x_m,y_1,\ldots,y_n]$,  \begin{enumerate}
	\item  $\rho(IJ) =\max\{\rho(I),\rho(J)\}$.
	\item  $\rho_a(IJ) =\max\{\rho_a(I),\rho_a(J)\}$.
\end{enumerate}
\end{proposition}
\begin{proof}
(1)
First, observe that $IJ= I\cap J$. Therefore, $(I \cap J)^t =(IJ)^t=I^tJ^t =I^t \cap J^t$, and $(IJ)^{(s)}=(I\cap J)^{(s)} = I^{(s)}\cap J^{(s)}=I^{(s)}J^{(s)}$ for all $s ,t \in \mathbb{N}$. Thus, by \Cref{int-res}, 
\[\rho(IJ) \leq \max\{\rho(I), \rho(J)\} \text{ and } \rho_a(IJ) \leq \max\{\rho_a(I), \rho_a(J)\}.\]
Now, let $s, t \in \mathbb{N}$ be such that $I^{(s)} \not\subset I^t$. Let $f \in I^{(s)} \setminus I^t$. If $fJ^{(s)} \subset (IJ)^t$, then $f \in (IJ)^t : J^{(s)} \subset I^t : J^{(s)} = I^t,$ a contradiction. Therefore, $fJ^{(s)} \not\subset (IJ)^t$ which implies that $(IJ)^{(s)} \not\subset (IJ)^t$. Thus $\rho(I) \leq \rho(IJ)$. Similarly, $\rho(J) \leq \rho(IJ)$. Hence $\max\{\rho(I),\rho(J)\} \leq \rho(IJ)$. 
\par (2) Let $s, t \in \mathbb{N}$ be such that $I^{(sp)} \not\subset I^{tp}$ for $p \gg 0$. Let $f_p \in I^{(sp)} \setminus I^{tp}$.  Then, $f_pJ^{(sp)} \not\subset (IJ)^{tp}$, and therefore, $(IJ)^{(sp)} \not\subset (IJ)^{tp}$.   Thus, we have $\rho_a(I) \leq \rho_a(IJ)$. Similarly, $\rho_a(J) \leq \rho_a(IJ)$, and hence, $\max\{\rho_a(I),\rho_a(J)\} \leq \rho_a(IJ)$.
Hence, the assertion follows.
\end{proof}

In \cite{SHJ21}, the authors prove that if $I$ and $J$  are nonzero proper homogeneous ideals in polynomial rings with different set of variables, then $\max\{\rho(I), \rho(J)\} \leq \rho(I+J) \leq \rho(I) + \rho(J)$. It is interesting to see if any of these two inequalities can be equality. We first give some sufficient conditions for the first inequality to be equality.

\begin{theorem}\label{res-sum1}
Let $I \subset \kk[x_1,\ldots,x_m]$ and $J \subset \kk[y_1,\ldots,y_n]$ be nonzero proper  homogeneous ideals. If $I^s=I^{(s)}$ for all $s \geq 1$, then $\rho(I+J) =\rho(J)$.
\end{theorem}
\begin{proof}
Since $I^s=I^{(s)}$ for all $s \geq 1$, $\rho(I)=1$. Therefore, it follows from \cite[Theorem 2.7]{SHJ21} that $\rho(J) \leq \rho(I+J)$. Now let $s, t \in \mathbb{N}$ be such that $\displaystyle \frac{s}{t} > \rho(J)$. Since $\rho(J) \geq 1$, we have $\displaystyle \frac{s-i}{t-i} \geq \frac{s}{t}$ for all $ 1 \leq i < t \leq s$. Therefore, $J^{(s-i)} \subset J^{t-i}$ for all $0 \leq i < t$. Using \cite[Theorem 3.4]{HNTT}, we get
\begin{align*}
\displaystyle (I+J)^{(s)} &= \sum_{i=0}^s I^{(i)}J^{(s-i)}= \sum_{i=0}^s I^{i}J^{(s-i)}\\ &= \sum_{i=0}^{t-1} I^{i}J^{(s-i)}+\sum_{i=t}^s I^{i}J^{(s-i)}\\ & \subset   \sum_{i=0}^{t-1} I^{i}J^{t-i}+I^t = (I+J)^t,
\end{align*}
Therefore, $\rho(J)$ is an upper bound for the set $\displaystyle \Big\{ \frac{s}{t} : s,t \in \mathbb{N} \text{ and } (I+J)^{(s)}\not\subset (I+J)^t\Big\}$. Hence, $\rho(J) = \rho(I+J)$.
\end{proof}
For $1\leq i \leq k$, let $I_i \subset \kk[x_{i,1}, \ldots, x_{i,m_i}]$ be nonzero proper homogeneous ideals. Then $I_1 + \cdots + I_k$ denotes the ideal generated by $I_1, \ldots, I_k$ in the polynomial ring $\kk[x_{1,1}, \ldots, x_{1,m_1}, \ldots, x_{k,1}, \ldots, x_{k,m_k}]$. Using the previous theorem, we can inductively extend it to the following result. 
\begin{corollary}\label{res-sum2}
Let $I_1,\ldots,I_k$ be nonzero proper homogeneous ideals in polynomial rings  $\kk[x_{1,1},\ldots,x_{1,m_1}],\ldots, \kk[x_{k,1},\ldots,x_{k,m_k}]$, respectively. If $I_j^{(s)}=I_j^s$ for $1\leq j\leq p\leq k$ and $s\geq 1$, then $\rho(I_1+\cdots+ I_k)= \rho(I_{p+1}+\cdots+I_k)$.
\end{corollary}
Now we deal with ideals whose symbolic Rees algebra is different from the Rees algebra. When a collection of ideals have resurgence equal to $1$, we obtain a formula for the resurgence of their sum. We first prove a technical result.

\begin{lemma}\label{tech-res}
Let $I_1,\ldots,I_k$ be nonzero proper homogeneous ideals in polynomial rings  $\kk[x_{1,1},\ldots,x_{1,m_1}],\ldots, \kk[x_{k,1},\ldots,x_{k,m_k}]$, respectively. If $I_i^{(p_i)} \not\subset I_i^{r_i}$ for $1 \leq i \leq k$, then $$(I_1+\cdots+I_k)^{(p_1+\cdots+p_k)} \not\subset (I_1+\cdots+I_k)^{r_1+\cdots+r_k-k+1}.$$
\end{lemma}
\begin{proof}
We prove this by induction on $k\geq 2$.  The result is true for $k=2$, by \cite[Lemma 3.3]{SHJ21}. Assume that $k>2$. Set $J = I_1+\cdots+I_{k-1}$. Then, by induction, $J^{(p_1+\cdots+p_{k-1})}=(I_1+\cdots+I_{k-1})^{(p_1+\cdots+p_{k-1})} \not\subset (I_1+\cdots+I_{k-1})^{r_1+\cdots+r_{k-1}-k+2}=J^{r_1+\cdots+r_{k-1}-k+2}.$ Now, by \cite[Lemma 3.3]{SHJ21},  $(I_1+\cdots+I_k)^{(p_1+\cdots+p_{k})}=(J+I_k)^{(p_1+\cdots+p_{k})}\not\subset (J+I_k)^{r_1+\cdots+r_k-k+1}=(I_1+\cdots+I_k)^{r_1+\cdots+r_k-k+1}$. Hence, the assertion follows.
\end{proof}

In \Cref{res-sum1}, we studied the resurgence when $I_i^{(s)} = I_i^s$ for some $i$. We now deal with the case when $I_i^{(s)} \neq I_i^s$  and $\rho(I_i) = 1$ for all $i$. It is to be noted that $\rho(I) = 1$ need not necessarily imply that $I^{(s)} = I^s$ for all $s \geq 1$, as you can  see in the example given in \cite[Remark 5.4]{DD20}.

\begin{theorem}\label{sum-resurgence}
Let $I_1,\ldots,I_k$ be nonzero proper homogeneous ideals in polynomial rings  $\kk[x_{1,1},\ldots,x_{1,m_1}],\ldots, \kk[x_{k,1},\ldots,x_{k,m_k}]$, respectively. For $1 \leq i \leq k$, let $p_i$ be the least positive  integer such that $I_i^{(p_i)} \neq I_i^{p_i}$. Assume that $p_1 \leq \cdots \leq p_k$. 
If $\rho(I_i)=1$ for all $1 \leq i \leq k $, then   $$\displaystyle \rho(I_1 + \cdots+I_k) =\max\Big\{ \frac{p_1+\cdots+p_r}{p_1+\cdots+p_r-r+1}: 2 \leq r \leq k \Big\}.$$
\end{theorem}
\begin{proof}
Set $$\alpha = \max\Big\{ \frac{p_1+\cdots+p_r}{p_1+\cdots+p_r-r+1}: 2 \leq r \leq k \Big\}.$$ Since $\rho(I_i)=1$ for all $1 \leq i \leq k$, $I_i^{(s)} \subset I_i^{s-1}$ for all $s \geq 2$ and for all $1 \leq i \leq k$.    
 For all $s \geq 2$, it follows from \cite[Theorem 3.4]{HNTT} that \begin{align*}
    (I_1+\cdots+I_k)^{(s)}  = \sum_{a_i \geq 0, \; a_1+\cdots+a_k=s} \Big(\prod_{i=1}^k I_i^{(a_i)}\Big). 
\end{align*} If $2 \leq s <p_1$, then  for all $a_1,\ldots,a_k \geq 0$ with $a_1+\cdots+a_k=s$,  $\prod_{i=1}^k I_i^{(a_i)} = \prod_{i=1}^k I_i^{a_i}$. Thus, $(I_1+\cdots+I_k)^{(s)} =(I_1+\cdots+I_k)^s$ for $1 \leq s <p_1.$ Now suppose $p_1+ \cdots +p_r \leq s < p_1+\cdots+p_{r+1}$ for some $1 \leq r \leq k-1$ or $p_1+ \cdots +p_k \leq s$. Then, for all $a_1,\ldots,a_k \geq 0$ with $a_1+\cdots+a_k=s,$  \begin{align*}
   \prod_{i=1}^k I_i^{(a_i)} &= \prod_{a_i \geq p_i} I_i^{(a_i)} \prod_{a_i<p_i} I_i^{(a_i)} \\& = \prod_{a_i \geq p_i} I_i^{(a_i)} \prod_{a_i<p_i} I_i^{a_i}\\& \subset \prod_{a_i \geq p_i} I_i^{a_i-1} \prod_{a_i<p_i} I_i^{a_i}. 
\end{align*} If $p_1+ \cdots +p_r \leq s < p_1+\cdots+p_{r+1}$, then $|\{a_i : a_i \geq p_i\}| \leq r$, and if $p_1+ \cdots +p_k \leq s$, then $|\{a_i : a_i \geq p_i\}| \leq k$. Therefore,   $\prod_{i=1}^k I_i^{(a_i)} \subset (I_1+\cdots+I_k)^{s-r}.$ Thus,

\begin{equation}\label{eqn1}
\begin{array}{lll}
(I_1+\cdots+I_k)^{(s)} \subset (I_1+\cdots+I_k)^{s-r} & \text{ if } & p_1+ \cdots +p_r \leq s < p_1+\cdots+p_{r+1}, \\
(I_1+\cdots+I_k)^{(s)} \subset (I_1+\cdots+I_k)^{s-k} & \text{ if } & p_1+ \cdots +p_k \leq s.
\end{array}
\end{equation}

Now, let $s,t \in \mathbb{N}$ be such that $(I_1+\cdots+I_k)^{(s)} \not\subset (I_1+\cdots+I_k)^{t}.$ If $2 \leq s <p_1$, then $t \geq s+1$, and hence, $\displaystyle\frac{s}{t} < 1 \leq \alpha.$ If $p_1+\cdots+p_r \leq s < p_1+\cdots+p_{r+1}$, then by \eqref{eqn1}, $t \geq s-r+1$, and hence, $\displaystyle \frac{s}{t} \leq \frac{s}{s-r+1} \leq \frac{p_1+\cdots+p_r}{p_1+\cdots+p_r-r+1} \leq \alpha. $ If $s \geq p_1+\cdots+p_k$, then by \eqref{eqn1}, $\displaystyle\frac{s}{t} \leq \frac{s}{s-k+1} \leq \frac{p_1+\cdots+p_k}{p_1+\cdots+p_k-k+1} \leq \alpha.$ Thus, $\alpha$ is an upper bound for the set $\displaystyle \Big\{\frac{s}{t} : s,t \in \mathbb{N} \text{ and } (I_1+\cdots+I_k)^{(s)} \not\subset (I_1+\cdots+I_k)^{t}\Big\}.$ Therefore $\rho(I_1+\cdots+I_k) \leq \alpha. $ Now, for $1 \leq r \leq k$, it follows from Lemma \ref{tech-res} that  \[(I_1+\cdots+I_r)^{p_1+\cdots+p_r} \not\subset (I_1+\cdots+I_r)^{p_1+\cdots+p_{r}-r+1}.\] Therefore,  $\displaystyle\rho(I_1+\cdots+I_r) \geq \frac{p_1+\cdots+p_r}{p_1+\cdots+p_r-r+1}.$
The assertion now follows form \cite[Theorem 2.7]{SHJ21}. 
\end{proof}

\begin{obs}{\em
\begin{enumerate}
\item In \cite[Conjecture 3.8]{SHJ21}, the authors conjectured that there exists a homogeneous ideal $I$ such  that $\rho(I^{[k]}) \to \infty$ if $k  \to \infty$  (refer to \cite{SHJ21} for the definition of $I^{[k]}$). If $I^{(s)} = I^s$ for all $s \geq 1,$ then by \cite[Theorem 3.4]{HNTT}, $(I^{[k]})^{(s)} = (I^{[k]})^s$ for all $s \geq 1$. If $\rho(I) = 1$ and $p$ is the least integer with $I^{(p)} \neq I^p$, then by \Cref{sum-resurgence}, $\displaystyle\rho(I^{[k])}) = \frac{2p}{2p-1}$. Therefore, for the conjecture to be true, it is necessary that $\rho(I) > 1.$ 
    
\item In fact, we have a stronger observation here. If $\rho(I_j) = 1$ for $1 \leq j \leq k$, then it follows from \Cref{res-sum2} and  \Cref{sum-resurgence}  that $\rho(I_1 + \cdots + I_k) \leq 2.$
\end{enumerate}
}
\end{obs}

For two ideals, $I$ and $J$, the containment $I \subset J$ does not really force any implication on their resurgences. It is interesting to ask what additional hypothesis on $I$ and/or $J$ can imply a relation between $\rho(I)$ and $\rho(J)$. In the following result, we come up with one such instance.

\begin{proposition}\label{res-lower}
Let $I$ be a squarefree monomial ideal in a polynomial ring $R$. Let $m \in R \setminus I$ be a squarefree monomial and  $J=I:m$. Then, $\rho(J) \leq \rho(I)$. 
\end{proposition}
\begin{proof}
We first prove the assertion when $m$ is a variable, say $x$.
We claim that for all $t \in \mathbb{N}$, $I^t : x^{\infty}=J^t$. For $u_1, \ldots, u_t \in J$, $u_1\cdots u_t \in I^t : x^t$ and hence $J^t \subset I^t : x^t \subset I^t : x^\infty$.
Now, let $w$ be a monomial in $I^t : x^{\infty}$. Then,  $wx^k \in I^t$ for some $k \in \mathbb{N}$. Write $wx^k = w_1\cdots w_tv'$ for  some minimal monomial generators $w_1,\ldots, w_t$ of $I$ and a monomial $v'$ in $R$. Let $1 \leq s \leq t$ be  such  that  $x$ does not divide $w_i$ for $i \leq s$ and $x$ divides $w_{j}$ for $j > s$. Then $w_i \in J$ for $i \leq s$ and $\displaystyle\frac{w_j}{x}\in J$ for $j > s$. Hence $\displaystyle w=\frac{w_1\cdots w_tv'}{x^k} = (w_1 \cdots w_s)\Big(\frac{w_{s+1}}{x} \cdots \frac{w_t}{x}\Big) \frac{v'}{x^{k+s-t}} \in J^t$. Therefore,  $I^t: x^{\infty} = J^t$ for all $t \in \mathbb{N}$. Now, let $s,t \in \mathbb{N}$ be such that $I^{(s)} \subset I^t$. Then, $I^{(s)} :x^{\infty} \subset I^t: x^{\infty}=J^t$. By the proof of \cite[Lemma 2.21]{MCM19}, we have  $I^{(s)} :x^{\infty} =J^{(s)}$ which implies that $J^{(s)} \subset J^t$. This proves that $\rho(J) \leq \rho(I)$.

If $m = x_1\cdots x_r$, then the assertion follows by observing that $J=I : m = ( \cdots ((I : x_1)  : x_2)): \cdots : x_r)$.
\end{proof}

For a squarefree monomial ideal $I \subset \kk[x_1, \ldots, x_\ell]$ of big height $h$, it is known that $\displaystyle \rho_a(I) \leq h - \frac{1}{\ell}$, \cite[Corollary 4.20]{DD20}.
Using the hypergraph language, we are able to refine this upper bound. Any squarefree monomial ideal can be viewed as the cover ideal of a hypergraph, \cite{HV08}. Recall that for a hypergraph $\HH$, $\chi(\HH)$ denotes the chromatic number of $\HH$.
\begin{theorem}\label{general-upperbound}
Let $\mathcal{H}$ be a hypergraph 
and $h$ denote the big height of $J(\HH)$. Then, $\displaystyle J(\mathcal{H})^{(rh-h)} \subset J(\mathcal{H})^r$ for all $\displaystyle r \geq \chi(\mathcal{H}).$ 
In particular, $$\displaystyle \rho_a(J(\mathcal{H})) \leq h -\frac{1}{\chi(\mathcal{H})}.$$ 
\end{theorem}
\begin{proof}
Let $V(\mathcal{H})=\{x_1,\ldots,x_n\}$.
Let $\pi(J(\HH))$ denote the least common multiple of the generators of $J(\HH).$ Note that every vertex of $\HH$ is a part of at least one vertex cover of $\HH$. Consequently,  $\displaystyle\pi(J(\HH)) = x_1 \cdots x_n$. By \cite[Theorem 3.2]{FHV11},  $\displaystyle (x_1 \cdots x_n)^{r-1} \in J(\HH)^r$ if and only if $r \geq \chi(\HH).$ Thus,  $\displaystyle \pi(J(\HH))^{r-1} \in J(\HH)^r$ for $r \geq \chi(\HH)$. Now, one can simply follow the same steps as in the proof of \cite[Proposition 4.15]{DD20} to get $J(\HH)^{(hr-h)} \subset J(\HH)^r$ for all $r \geq \chi(\HH)$. 
Now it follows from \cite[Theorem 4.5]{DD20} that $\displaystyle  \rho_a(J(\HH)) \leq h - \frac{1}{\chi(\HH)}$.
\end{proof}

\begin{remark}
{\em
It may be noted that given a squarefree monomial ideal $I$, one can compute $\chi(\HH)$ without constructing the associated hypergraph $\HH$, \cite[Theorem 3.2]{FHV11}. Therefore, one can view \Cref{general-upperbound} as an algebraic upper bound for the asymptotic resurgence of squarefree monomial ideals in terms of algebraic invariants. It can also  be seen that from the containment given in \Cref{general-upperbound}, one can get a possibly weaker upper bound for the  resurgence. From the containment in \Cref{general-upperbound} and \cite[Theorem 2.5, Remark 2.7]{G20}, we get $\rho(J(\HH)) \leq \displaystyle h - \frac{h}{m\chi(H)}$ for some $m \in \mathbb{N}$. 
}
\end{remark}

Given a hypergraph $\HH$, another hypergraph $\HH'$ is said to be a  subhypergraph of $\HH$ if $V(\HH') \subseteq V(\HH)$, $E(\HH') \subset E(\HH)$ and for any $e \in E(\HH)$, $e \in E(\HH')$ if and only if $V(e) \subset V(\HH')$. We conclude this section  by interpreting \Cref{res-lower} in terms of hypergraph theory.
\begin{proposition}\label{induced-asym-resurgence}
Let $\mathcal{H}$ be a hypergraph and let $\mathcal{H}'$ be a subhypergraph of $\HH$. Then 
\begin{enumerate}
    \item $\displaystyle \rho(J(\HH')) \leq \rho(J(\HH)),$
    \item $\displaystyle \rho_a(J(\HH')) \leq \rho_a(J(\HH)).$
\end{enumerate}
\end{proposition}
\begin{proof}
For $U\subset V(\HH)$, a non-empty subset, set $\displaystyle x_U = \prod_{x \in U} x$ and $\displaystyle J(\HH)_U= J(\HH): x_U$.
We claim that for $\displaystyle U= V(\HH)\setminus V(\HH')$, $\displaystyle J(\HH)_U=J(\HH')$. If $\Gamma$ is a vertex cover of $\HH$, then $\Gamma \cap V(\HH')  = \Gamma \setminus U$ is a vertex cover of $\HH'$ which implies that $\displaystyle J(\HH)_U \subset J(\HH')$. Also, if $\Gamma$ is a vertex cover of $\HH'$, then $\Gamma \cup U$ is a vertex cover of $\HH$. Therefore, $J(\HH') \subset J(\HH)_U$, and hence,  $\displaystyle J(\HH)_U=J(\HH')$. 

\par (1) Since $\displaystyle J(\HH') = J(\HH) : x_U$, where $\displaystyle U=V(\HH)\setminus V(\HH')$, by \Cref{res-lower}, $\displaystyle\rho(J(\HH)) \geq \rho(J(\HH):x_U)=\rho(J(\HH'))$. 
\par (2)  It follows from \cite[Corollary 2.24]{MCM19} that $\displaystyle\rho_a(J(\HH)) \geq \max \{ \rho_a(J(\HH)_A): A \subset V(\HH), A \neq \emptyset\}.$ Therefore, $\displaystyle\rho_a(J(\HH)) \geq \rho_a(J(\HH)_U) =\rho_a(J(\HH'))$. 
\end{proof}
The main utility of the above result is that it allows one to construct a lower bound for the resurgence and asymptotic resurgence. We will see concrete applications of this result in the following section.

\section{resurgence of cover ideals of graphs}
In this section, we study the resurgence of cover ideals of graphs.  In \cite{HHTV}, it was proved that the Rees algebra of the cover ideal of a finite simple graph is generated in degree at most $2$. We begin with an observation that follows  mainly from \cite[Theorem 5.1]{HHTV}.

\begin{obs}\label{walds}
Let $G$ be a graph.
\begin{enumerate}
\item For any $s \geq 1$, $\displaystyle J(G)^{(2s)} = (J(G)^{(2)})^s$. In particular, $\displaystyle\alpha(J(G)^{(2s)})=s \alpha(J(G)^{(2)}).$
\item For any $s \geq 1$, $\displaystyle J(G)^{(2s+1)} = J(G)(J(G)^{(2)})^s$. In particular, $\displaystyle\alpha(J(G)^{(2s+1)})=\alpha(J(G))+ s \alpha(J(G)^{(2)}).$
	\item $\displaystyle \hat{\alpha}(J(G)) = \frac{\alpha(J(G)^{(2)})}{2} \leq \frac{|V(G)|}{2}.$
\end{enumerate}
\end{obs}
\begin{proof}
It follows from \cite[Theorem 5.1]{HHTV} that the symbolic Rees algebra is generated in degree two.  Therefore, for every $s \geq 1$, we have $\displaystyle J(G)^{(2s)}=(J(G)^{(2)})^s$, and $\displaystyle J(G)^{(2s+1)}=J(G)(J(G)^{(2)})^s$. Thus, for any $s \geq 1$, $\displaystyle\alpha(J(G)^{(2s)})=s \alpha(J(G)^{(2)}),$ and 
$\displaystyle\alpha(J(G)^{(2s+1)})=\alpha(J(G))+ s \alpha(J(G)^{(2)}).$ This proves $(1)$ and $(2)$. The equality in $(3)$ follows directly from $(1)$ and $(2).$ Observe that $\displaystyle J(G)^{(2)} = \bigcap_{\{x,y\} \in E(G)}  (x,y)^2.$ Since $\displaystyle \prod_{x \in V(G)} x \in (x,y)^2$ for every $\{x,y\} \in E(G),$ we have  $\displaystyle \prod_{x \in V(G)} x \in J(G)^{(2)}.$ Hence $\alpha(J(G)^{(2)}) \leq |V(G)|.$ This  proves the inequality in $(3)$.
\end{proof}
It may be note that the equality in (3) has already been proved in \cite[Corollary 4.4]{DG20}. Also, in \cite[Remark 4.10]{DG20}, it is proved that if the sdefect$(J(G),2) = 1$, then the product of all variables is contained in $J(G)^{(2)}$. In the proof of \Cref{walds}, we have proved this conclusion without the assumption on symbolic defect.

The following result ensures that to study the resurgence/asymptotic resurgence of cover ideals of graphs it is enough to study those of cover ideals connected graphs.
\begin{proposition}
Let $G$  be a disconnected graph with non-trivial connected components  $\displaystyle G_1, \ldots,G_k$. Then : \begin{enumerate}
	\item $\displaystyle\rho(J(G)) = \max \{ \rho(J(G_i)) :  1\leq i \leq k\}$.
	\item $\displaystyle\rho_a(J(G)) = \max \{ \rho_a(J(G_i)) :  1\leq i \leq k\}$.
\end{enumerate}
\end{proposition}

\begin{proof}
Note that $\Gamma$ is a vertex cover of $G$ if and only if $\Gamma \cap V(G_i)$ is a vertex cover of $G_i$ for each $i$. Therefore, 
$\displaystyle J(G)=J(G_1)\cdots J(G_k)$. Now, the assertion follows from \Cref{dis-res}.
\end{proof}

In \cite[Question 2.2]{G20},  Grifo considered a stable version of Harbourne's conjecture and asked if $I$ is a radical ideal of big height $h$ in a regular ring $R$, then for a given $C > 0$, does there exist $N$ such that $I^{(hr-C)} \subset I^r$ for all $r \geq N$? It follows from \cite[Remark 2.7]{G20} and \cite[Corollary 4.20]{DD20} that for the class of squarefree monomial ideals, the question has an affirmative answer.  In \cite[Proposition 4.15]{DD20}, for a squarefree monomial ideal $I \subset \kk[x_1, \ldots, x_\ell]$ of big height $h$, the authors prove that $N = \ell$ works for $C = h$. In the next result, for the class of cover ideals we explicitly obtain an $N$, for a given $C$, such that $J(G)^{(2r-C)}\subset J(G)^r$ for all $r \geq N$.

\begin{proposition}\label{res-upper}
Let $G$  be a graph. Then for any positive integer $c$,
\begin{enumerate}
    \item $J(G)^{(2r-2c)} \subset J(G)^r$ for every $\displaystyle r \geq c\chi(G).$
    \item $J(G)^{(2r-2c-1)} \subset J(G)^r$ for every $\displaystyle r \geq c\chi(G)+1.$
\end{enumerate} 
\end{proposition}
\begin{proof}
\par (1) If $c=1$, then the  assertion follows from \Cref{general-upperbound}. Assume that $c \geq 2$. For convenience, write $\chi(G) = k$. Then, 
\begin{align*}
    J(G)^{(2ck-2c)} &=\Big(J(G)^{(2)}\Big)^{ck-c} & \text{(By \Cref{walds})} \\
    & = \Big(\Big(J(G)^{(2)}\Big)^{k-1} \Big)^c & \text{(By \Cref{walds})} \\
    &=\Big(J(G)^{(2k-2)}\Big)^{c} & \text{(By \Cref{walds})} \\
   & \subset \Big(J(G)^k\Big)^c & \text{(By \Cref{general-upperbound})} \\
    & =J(G)^{ck}.
\end{align*}
It also follows from \Cref{walds} that  $\displaystyle J(G)^{(s+2)}=J(G)^{(2)} J(G)^{(s)} \subset J(G) J(G)^{(s)}$ for any $s \geq 1$. Thus, by \cite[Discussion 2.10]{G20},  $\displaystyle J(G)^{(2r-2c)} \subset J(G)^{r}$ for every $r \geq c\chi(G)$.

\par (2) For $ r \geq c \chi(G)+1,$ we get 
\begin{align*}
    J(G)^{(2r-2c-1)} &=J(G)^{(2(r-1)-2c)} J(G) \subset J(G)^{r-1} J(G) = J(G)^{r},
\end{align*} where the  containment follows from part $(1)$.
\end{proof}

\begin{remark}{\em
 For a graph $G$, if $\chi(G) = \omega(G),$ then $c \chi(G)$ acts as the minimum of the set $\{t  :  J(G)^{(2r-2c)}\subset J(G)^r \text{ for all } r \geq t\}$ and $c \chi(G)+1$ acts as the minimum of the set $\{t  :  J(G)^{(2r-2c-1)}\subset J(G)^r \text{ for all } r \geq t\}$. Let $H$ be an induced subgraph of $G$ which is a complete graph of size $\omega(G)  =\chi(G)$. Then, it follows from  \cite[Theorem A]{LM15} that $J(H)^{(2r-2c)} \not\subset J(H)^r$ for all $r < c\chi(G)$ and $J(H)^{(2r-2c-1)} \not\subset J(H)^r$ for all $r < c\chi(G)+1$. If $J(G)^{(2r-2c)} \subset J(G)^r$ for some $r < c\chi(G)$, then recursively  applying  \cite[Lemma 2.21]{MCM19} and the proof of \Cref{res-lower}, we get that  
\[ J(H)^{(2r-2c)}  \subset J(H)^r,\] which is a contradiction. Therefore, $J(G)^{(2r-2c)} \not\subset J(G)^r$ for all $r < c\chi(G)$. Similarly, $J(G)^{(2r-2c-1)} \not\subset J(G)^r$ for all $r < c\chi(G)+1$.
}
\end{remark}

Using \Cref{res-upper} and \cite[Theorem 4.5]{DD20}, one can derive that $\displaystyle \rho_a(J(G)) \leq 2 - \frac{2}{\chi(G)}$. In the next result, we prove that the same upper bound holds for the resurgence as well. In the following theorem, $\alpha(G)$ denotes the \textit{independence number of a graph} $G$ and it is defined to be the maximum cardinality over independent sets of $G$. It is to be noted that $\alpha(G)$ is  different from $\alpha(J(G))$. While the first one is a combinatorial invariant, the second one is an algebraic invariant.

\begin{theorem}\label{res-upper1}
Let $G$ be a  connected graph on $n$ vertices. Then
\[ \max\left \{ 2-\frac{2}{\omega(G)}, 2 - \frac{2 \alpha(G)}{n} \right\} \leq \rho_a(J(G)) \leq \rho(J(G)) \leq 2-\frac{2}{\chi(G)}.\] 
\end{theorem}

\begin{proof}
Let $U$ be a clique in $G$ of size $\omega(G)$. Set $H=G[U]$. Note that $H$ is a complete graph on $\omega(G)$ vertices. Then, it follows from  \cite[Theorem C]{LM15} that $\displaystyle \rho(J(H))=2-\frac{2}{\omega(G)} $.  By \cite[Theorem 2.10]{Vill08}, $J(H)$ is a normal ideal. Consequently, by \cite[Corollary 4.14]{MCM19}, $\displaystyle \rho_a(J(H))=\rho(J(H))=2-\frac{2}{\omega(G)} $. Thus, by \Cref{induced-asym-resurgence},  $\displaystyle 2 - \frac{2}{\omega(G)}= \rho_a(J(H)) \leq \rho_a(J(G))$.
Since the complement of any vertex cover is an independent set, it is easy to see that $\alpha(J(G)) = n - \alpha(G)$.  Therefore, using \Cref{walds}, we get $\displaystyle \frac{2(n-\alpha(G))}{n} \leq \frac{\alpha(J(G))}{\hat{\alpha}(J(G))} \leq \rho_a(J(G))$.

Next, we prove that $\displaystyle \rho(J(G)) \leq 2- \frac{2}{\chi(G)}.$ Let $s ,t \in \mathbb{N}$ be such that $J(G)^{(s)} \not\subset J(G)^t.$ Suppose that $s < 2(\chi(G)-1)$ and $s$ is even. Then $s =2r-2$ for some $2 \leq r < \chi(G).$ By \Cref{walds}, $J(G)^{(s)} =\Big( J(G)^{(2)} \Big)^{r-1} \subset J(G)^{r-1}.$ Consequently, $t \geq r$, and therefore, $\displaystyle \frac{s}{t} \leq \frac{2r-2}{r} \leq 2 - \frac{2}{\chi(G)}$. 
Next suppose that $s < 2(\chi(G)-1)$ and $s$ is odd. Then $s =2r-1$ for some $2 \leq r < \chi(G).$ By \Cref{walds}, $J(G)^{(s)}=J(G)^{(2r-1)}= (J(G)^{(2r-2)})J(G) \subset J(G)^{r}$.  Consequently, $t \geq r+1$ so that $\displaystyle \frac{s}{t} \leq \frac{2r-1}{r+1} = 2 - \frac{3}{r+1} < 2 - \frac{2}{r+1} \leq 2 - \frac{2}{\chi(G)}$.  Assume now that  $s \geq  2(\chi(G)-1).$  Write $s =q(2\chi(G)-2)+r$ for some $q \in \mathbb{N} $ and $0 \leq r < 2 \chi(G) -2.$ Then, 
\begin{align*}
    J(G)^{(s)} & = J(G)^{(q(2\chi(G)-2)+r)} & \\  & =\Big(J(G)^{(2)}\Big)^{q(\chi(G)-1)} J(G)^{(r)} & \text{(By \Cref{walds})} \\ & = \Big(J(G)^{(2\chi(G)-2)}\Big)^{q} J(G)^{(r)} & \text{(By \Cref{walds})} \\
    & \subset  \Big(J(G)^{\chi(G)}\Big)^{q}J(G)^{(r)} & \text{(By  \Cref{res-upper})}\\ 
     & \subset
    J(G)^{q \chi(G) + \ceil[\big]{\frac{r}{2}}} & \text{(By the  previous  paragraph)}.
\end{align*}
Thus, $ t \geq q \chi(G) + \ceil{\frac{r}{2}} +1$, and hence, $\displaystyle \frac{s}{t} \leq \frac{q(2\chi(G)-2)+r}{q \chi(G) + \ceil[\big]{\frac{r}{2}} +1} \leq 2- \frac{2}{\chi(G)}$. 
Therefore, $\displaystyle \rho(J(G)) \leq 2 - \frac{2}{\chi(G)}.$
\end{proof}

It may be noted that the invariants $\omega(G)$ and $\dfrac{|V(G)|}{\alpha(G)}$ are incomparable. If $G = C_{2n+1}$, then $\omega(G) = 2$ and $\dfrac{|V(G)|}{\alpha(G)} = 2+ \dfrac{1}{n}$. Now, if $G=K_{1,n}$, then $\omega(G) = 2$ and $\dfrac{|V(G)|}{\alpha(G)} = 1 + \dfrac{1}{n}$. 

We now list out some immediate consequence of \Cref{res-upper1} to get refined bounds and explicit expressions for  the resurgence and asymptotic resurgence of cover ideals of some important classes of graphs.

\begin{corollary}\label{res-upper2}
Let $G$ be a graph. 
\begin{enumerate}
\item If $\chi(G)=\omega(G)$, then $\displaystyle \rho_a(J(G))= \rho(J(G))= 2-\frac{2}{\omega(G)}.$ In particular, for any perfect graph we know resurgence and asymptotic resurgence. 
\item If $G$ is a chordal graph, then $\displaystyle \rho_a(J(G)) = \rho(J(G)) = 2 - \frac{2}{\omega(G)}$.
\item If $\displaystyle \chi_f(G) = \frac{|V(G)|}{\alpha(G)}$ (e.g. vertex transitive graphs), then $$\displaystyle 2 - \frac{2}{\chi_f(G)} \leq \frac{\alpha(J(G))}{\hat{\alpha}(J(G))} \leq \rho_a(J(G)) \leq \rho(J(G)) \leq  2- \frac{2}{\chi(G)}.$$ Moreover, if $\chi(G)=\chi_f(G)$, then  $\displaystyle \frac{\alpha(J(G))}{\hat{\alpha}(J(G))}=\rho_a(J(G)) =\rho(J(G)) = 2- \frac{2}{\chi(G)}.$ 
\end{enumerate}
\end{corollary}
\begin{proof} (1)
Since $\omega(G) =\chi(G)$,  by \Cref{res-upper1}, $\displaystyle \rho(J(G)=\rho_a(J(G))= 2 - \frac{2}{\omega(G)}$.

\par (2) Since a chordal graph is a perfect graph, by $(1)$, the assertion follows. 

\par (3) The first inequality follows from (the second paragraph of) the proof of \Cref{res-upper1}. The rest of the inequalities follows from the statement of \Cref{res-upper1}.
\end{proof}

It may be noted that \Cref{res-upper1} enables us to compute or obtain a tight bound for the resurgence of cover ideals of several important classes of graphs such as perfect graphs (bipartite graphs, chordal graphs, cographs, permutation graphs, even-wheel graphs), Peterson graph, Cayley graphs, complete multipartite graphs. It has come to our attention that Grisalde, Seceleanu and Villarreal proved the lower bound in  \Cref{res-upper1} and \Cref{res-upper2}(1),
\cite{Villa22}.

Herzog, Hibi and Trung proved that $J(G)^{(s)} = J(G)^s$ for all $s \geq 1$ if and only if $G$ is a bipartite graph, \cite[Theorem 5.1]{HHTV}. Thus cover ideal of bipartite graphs have unit resurgence. Therefore, to study the resurgence of cover ideals of graphs with non-unit resurgence, one has to look for graphs containing odd cycles. We first deal with simplest such situation, namely $G = C_{2n+1}.$

\begin{theorem}\label{odd-cycle}
If $G=C_{2n+1}$, then 
\begin{enumerate}
	\item $\displaystyle \rho(J(G))=\rho_a(J(G))= \frac{\alpha(J(G))}{\hat{\alpha}(J(G))}= \frac{2n+2}{2n+1}.$
	\item $J(G)^{(2nt+2t)} \subset J(G)^{2nt+t}$ for all $t \geq 1.$
\end{enumerate}
\end{theorem}
\begin{proof}
(1) It follows from \cite[Proposition 5.3]{HHTV} that $\displaystyle J(G)^{(2)}=J(G)^2+ \Big(  \prod_{x \in V(G)} x\Big)$. Therefore, 
\[\alpha(J(G)^{(2)})=\min \{2n+1, \alpha(J(G)^2)\}=\min \{2n+1, 2  \alpha(J(G))\}.
\]
Note that every vertex cover of $G$ has at least $n+1$ elements, and $\{x_{2i-1}: 1 \leq i \leq n+1\}$ is a minimal vertex cover of $G$. Therefore, $\alpha(J(G)) = n+1$ and   $\alpha(J(G)^{(2)})=2n+1$.
By \Cref{walds}, $\displaystyle \hat{\alpha}(J(G))=\frac{\alpha(J(G)^{(2)})}{2}=n+\frac{1}{2}.$ 
Hence, by \cite[Theorem 1.2]{EBA}, $\displaystyle \frac{2n+2}{2n+1}= \frac{\alpha(J(G))}{\hat{\alpha}(J(G))} \leq \rho_a(J(G)) \leq \rho(J(G)).$ We now prove that $\displaystyle \rho(J(G)) \leq  \frac{2n+2}{2n+1}.$ 
\vskip 2mm \noindent
\textbf{Claim:} $\displaystyle  \mathfrak{m}J(G)^{(2)} \subset J(G)^2$ and $\displaystyle J(G)^{(2)} \subset \mathfrak{m}^nJ(G)$, where $\mathfrak{m}$ is the unique homogeneous maximal ideal in $R$. 

Since $\displaystyle J(G)^{(2)}=J(G)^2+ \Big(  \prod_{x \in V(G)} x\Big)$, it is enough to prove that $\displaystyle \mathfrak{m}\Big(  \prod_{x \in V(G)} x\Big) \subset J(G)^2$. For a $y \in V(G)$, $G\setminus y$ is a path graph on $2n$ vertices. Let $A, B  \subset V(G\setminus y)$ be such that $V(G\setminus y) =A \sqcup B$, and $A, B$ are independent sets. Note that $A$ and $B$ are vertex covers of $G\setminus y$. Therefore, $A \sqcup \{y\}$ and $B \sqcup \{y\}$ are vertex covers of $G$. Now, $\displaystyle y\prod_{x \in V(G)} x = \Big(\prod_{x \in A \sqcup \{y\}} x\Big)\Big( \prod_{x \in B \sqcup \{y\}} x \Big)  \in J(G)^2$. Hence, $\displaystyle  \mathfrak{m}J(G)^{(2)} \subset J(G)^2$.
\par Since $\alpha(J(G))=n+1$, $J(G) \subset \mathfrak{m}^{n+1}$. Therefore, $J(G)^2 \subset \mathfrak{m}^{n+1}J(G) \subset \mathfrak{m}^{n}J(G)$. Now, we show that $\displaystyle \prod_{x \in V(G)} x \in \m^n J(G)$.   Since $\{x_{2i-1}: 1 \leq i \leq n+1\}$ is a minimal vertex cover of $G$, we have  $\displaystyle \prod_{x \in V(G)} x =\Big(\prod_{i=1}^{n} x_{2i} \Big) \Big(\prod_{i=1}^{n+1}x_{2i-1}\Big) \in \mathfrak{m}^n J(G)$.  Hence, $\displaystyle J(G)^{(2)} \subset \mathfrak{m}^nJ(G)$. This completes the proof of the claim.
\par  Since $\mathcal{R}_s(J(G))$ is generated by linear and degree two forms, it follows from the Claim and \Cref{gen-GHM} that $\displaystyle \rho(J(G)) \leq \frac{2n+2}{2n+1}$. Hence $\displaystyle \rho(J(G)) = \frac{2n+2}{2n+1}$.
\par (2) By \Cref{walds}, $J(G)^{(2nt+2t)}= (J(G)^{(2n+2)})^t$ for all $t \geq 1$. Hence, the assertion follows from the Claim in (1) and \Cref{gen-GHM}.
\end{proof}

As a consequence, we characterize bipartite graphs in terms of resurgence and asymptotic resurgence, analogous to Theorem 5.1(b) of \cite{HHTV}. 

\begin{theorem}\label{rho=1}
Let $G$ be a  graph. Then 
\begin{enumerate}
\item[] $\rho(J(G)) =1$ if and only if $G$ is a bipartite graph if and only if $\rho_a(J(G))=1$.
\end{enumerate}
\end{theorem}

\begin{proof}
If $G$ is bipartite, then it follows from \cite[Theorem 5.1]{HHTV} that
$J(G)^t=J(G)^{(t)}$ for all $t \geq 1$. Thus, $\rho(J(G))
=\rho_a(J(G))=1$. 

Suppose now that  $G$ is a non-bipartite graph. Let $C$ be an induced odd-cycle in $G$.
By \Cref{induced-asym-resurgence}, $\rho(J(G)) \geq \rho_a(J(G)) \geq  \rho_a(J(C)) > 1$,
where the last inequality follows from  \Cref{odd-cycle}. This proves the assertion.
\end{proof}

%

We now study the resurgence and asymptotic resurgence of cover ideals of clique-sum of two graphs in terms of these invariants of the individual graphs.

\begin{theorem}\label{clique-sum}
Let $\displaystyle G=G_1 \cup G_2$ be  a clique-sum of $G_1$ and $G_2$. Then :
\begin{enumerate}
	\item For any $t \geq 1$, $\displaystyle J(G)^t=J(G_1)^t \cap J(G_2)^t$.
	\item For any $s \geq 1$, $\displaystyle J(G)^{(s)}=J(G_1)^{(s)} \cap J(G_2)^{(s)}$.
	\item $\displaystyle \rho(J(G))=\max\{\rho(J(G_1)),\rho(J(G_2))\}$.
	\item $\displaystyle \rho_a(J(G))=\max\{\rho_a(J(G_1)),\rho_a(J(G_2))\}$.
\end{enumerate}  
\end{theorem}

\begin{proof} Since $G$ is a clique-sum of $G_1$ and $G_2$, there exists an induced complete graph $K_r$ of $G$ such that $G_1 \cap G_2 =K_r$ and  $G_1, G_2 \neq K_r$. \par (1) Let $t $ be any positive integer. 
Since $J(G)= J(G_1) \cap J(G_2)$, $J(G)^t \subset J(G_1)^t \cap J(G_2)^t$. We now prove the reverse inclusion.
Assume, without loss of generality, that $V(K_r)=\{x_1,\ldots,x_r\}$. Suppose $r=1$. Let $u$ be a minimal monomial generator of $J(G_1)^t\cap J(G_2)^t$. Then, $\displaystyle u=\frac{u_1 u_2}{\gcd(u_1,u_2)}$, where $u_i$ is a minimal monomial generator of $J(G_i)^t$. Write $\displaystyle u_i =\prod_{j=1}^t m_{ij}$, where $m_{ij}$'s are minimal generators of $J(G_i)$. Let $k_1 \leq k_2$ be such that  $x_1 \mid m_{ij}$ for $1 \leq j \leq k_i$, and $ x_1 \nmid m_{ij}$ for $k_i < j \leq t$. Therefore, $\displaystyle \gcd(u_1,u_2)=x_1^{k_1}$. Now, \begin{align*}\displaystyle 
	u &= \frac{u_1 u_2}{\gcd(u_1,u_2)} \\ & = \frac{(\prod_{j=1}^t m_{1j}) (\prod_{j=1}^t m_{2j})}{x_1^{k_1}} \\&=
	\prod_{j=1}^{k_1} \Big(\frac{m_{1j}}{x_1} m_{2j} \Big) \cdot  \prod_{j=k_1+1}^t (m_{1j}m_{2j}).
	\end{align*}
Note that if $C_1$ and $C_2$  are vertex covers of $G_1$ and $G_2$ respectively such that $x_1 \in C_1 \cap C_2$, then $(C_1 \setminus \{x_1\}) \cup C_2$ is a vertex cover of $G$. Therefore, $\displaystyle \frac{m_{1j}}{x_1} m_{2j} \in J(G)$ for $1 \leq j \leq k_1$. Also, if $C_i$ is a vertex cover of $G_i$ for $i = 1, 2$, then $C_1 \cup C_2$ is a vertex cover of $G$, and hence, $m_{1j}m_{2j} \in J(G)$ for $k_1 +1 \leq j \leq t$.   Thus, $u \in J(G)^t$, and the result is true for $r=1$. 

Now assume that $r>1$ and the result is true for $r-1$, i.e., if $H$ is a clique-sum of $H_1$ and $H_2$ such that $\displaystyle H_1 \cap H_2 =K_{r-1}$, then $\displaystyle J(H_1)^t \cap J(H_2)^t \subset J(H)^t$ for all $t\geq 1$.  Let $u$ be a minimal monomial generator of $J(G_1)^t\cap J(G_2)^t$. Then, $\displaystyle u=\frac{u_1 u_2}{\gcd(u_1,u_2)}$, where $\displaystyle u_i =\prod_{j=1}^t m_{ij}$ for some minimal monomial generators  $m_{i1},\ldots,m_{it}$ of $J(G_i)$. Set $H_i =G_i \setminus x_r$ for $i=1,2$, and $H=G\setminus x_r$. Then, $H$ is a clique-sum of $H_1$ and $H_2$ such that $H_1 \cap H_2 =K_{r-1}$. Therefore, by induction, $J(H_1)^t \cap J(H_2)^t \subset J(H)^t$ for all $t \geq 1$. Let $l_1 \leq l_2$ be such that  $ x_r \mid m_{ij}$ for $1 \leq j \leq l_i$, and $ x_r \nmid m_{ij}$ for $l_i < j \leq t$. Note that, since $x_r \nmid m_{ij}$ for $l_i <j \leq t$, $x_1 \cdots x_{r-1} \mid m_{ij}$. Moreover, $\displaystyle \gcd\left(\frac{m_{1j}}{x_1\cdots x_{r-1}}, \frac{m_{2j}}{x_1\cdots x_{r-1}}\right) =1$ for $l_2 < j \leq t$. Define $m_{ij}' :=\frac{m_{ij}}{x_r}$, for $i \in \{1,2\}$ and $1 \leq j \leq l_i$. 
Set $\displaystyle u_1'= \Big(\prod_{j=1}^{l_1} m_{1j}'\Big)\cdot \Big(\prod_{j=l_1+1}^{l_2} m_{1j}\Big)$, and $\displaystyle u_2' = \prod_{j=1}^{l_2} m_{2j}'$. Therefore, we can write 
\begin{align*}
\displaystyle \gcd(u_1,u_2) & =(x_1 \cdots x_{r-1})^{t-l_2} \cdot \gcd\Big(\prod_{j=1}^{l_2} m_{1j},\prod_{j=1}^{l_2} m_{2j}\Big)\\&=(x_1 \cdots x_{r-1})^{t-l_2} \cdot x_r^{l_1} \cdot \gcd\Big(\frac{\prod_{j=1}^{l_2} m_{1j}}{x_r^{l_1}},\frac{\prod_{j=1}^{l_2} m_{2j}}{x_r^{l_1}}\Big)\\ & =(x_1 \cdots x_{r-1})^{t-l_2} \cdot x_r^{l_1} \cdot \gcd\Big(\frac{\prod_{j=1}^{l_2} m_{1j}}{x_r^{l_1}},\frac{\prod_{j=1}^{l_2} m_{2j}}{x_r^{l_2}}\Big)\\&= (x_1 \cdots x_{r-1})^{t-l_2} \cdot x_r^{l_1} \cdot \gcd(u_1',u_2').
\end{align*} Now,
\begin{align*}
u& = \frac{u_1u_2}{\gcd(u_1,u_2)}\\&= \frac{\Big(\prod_{j=1}^{l_2} m_{1j}m_{2j}\Big) \cdot \Big(\prod_{j=l_2+1}^t m_{1j}m_{2j}\Big)}{(x_1 \cdots x_{r-1})^{t-l_2} \cdot x_r^{l_1} \cdot \gcd(u_1',u_2')}\\&
=\frac{x_r^{l_1+l_2}u_1'u_2'}{x_r^{l_1} \cdot \gcd(u_1',u_2')} \cdot \frac{\prod_{j=l_2+1}^t m_{1j}m_{2j}}{(x_1 \cdots x_{r-1})^{t-l_2}}\\& =
x_r^{l_2} \cdot \frac{u_1'u_2'}{\gcd(u_1',u_2')} \cdot \prod_{j=l_2+1}^t\Big(\frac{ m_{1j}}{x_1 \cdots x_{r-1}}\Big)m_{2j}.
\end{align*}
Note that if $C_1$ and $C_2$  are vertex covers of $G_1$ and $G_2$ respectively such that $x_1, \ldots, x_{r-1} \in C_1 \cap C_2$, then $(C_1 \setminus \{x_1,\ldots,x_{r-1}\}) \cup C_2$ is a vertex cover of $G$. Therefore, $\displaystyle \frac{ m_{1j}}{x_1 \cdots x_{r-1}}m_{2j} \in J(G)$ for $l_2 < j \leq t$, and hence, $\displaystyle\prod_{j=l_2+1}^t\Big(\frac{ m_{1j}}{x_1 \cdots x_{r-1}}\Big)m_{2j} \in J(G)^{t-l_2}$. Next, observe that $u_i' \in J(H_i)^{l_2}$ for $i\in \{1,2\}$. Therefore, $\displaystyle \frac{u_1'u_2'}{\gcd(u_1',u_2')} \in J(H_1)^{l_2} \cap J(H_2)^{l_2} \subset J(H)^{l_2}.$ Since $\displaystyle \frac{u_1'u_2'}{\gcd(u_1',u_2')} \in J(H)^{l_2}$, there  exist minimal monomial generators $v_1,\ldots,v_{l_2}$ of $J(H)$ such that $\displaystyle \prod_{j=1}^{l_2} v_j \mid \frac{u_1'u_2'}{\gcd(u_1',u_2')}$. Then, $\displaystyle \prod_{j=1}^{l_2} (x_rv_j) \mid x_r^{l_2}\cdot \frac{u_1'u_2'}{\gcd(u_1',u_2')}$. Since $x_rv_j \in J(G)$ for $1 \leq j \leq l_2$, $\displaystyle x_r^{l_2}\cdot \frac{u_1'u_2'}{\gcd(u_1',u_2')} \in J(G)^{l_2}$. Thus, $u \in J(G)^t$, and therefore, $J(G_1)^t \cap J(G_2)^t \subset J(G)^t$. Hence $J(G)^t = J(G_1)^t \cap J(G_2)^t$.
\par(2) Follows from \Cref{remark-sympow}.
\par(3) By \Cref{induced-asym-resurgence}, $\max\{\rho(J(G_1)),\rho(J(G_2))\} \leq \rho(J(G))$ as  $G_i$ is an induced subgraph of $G$. By $(1)$, $(2)$ and  \Cref{int-res}, $\rho(J(G)) \leq \max\{\rho(J(G_1)),\rho(J(G_2))\}$. Hence, $\rho(J(G)) = \max\{\rho(J(G_1)),\rho(J(G_2))\}$. 
\par (4) By \Cref{induced-asym-resurgence}, $\max\{\rho_a(J(G_1)),\rho_a(J(G_2))\} \leq \rho_a(J(G))$ as  $G_i$ is an induced subgraph of $G$. By parts $(1-2)$ and \Cref{int-res},  $\rho_a(J(G)) \leq \max\{\rho_a(J(G_1)),\rho_a(J(G_2))\}$. Hence, $\rho_a(J(G)) = \max\{\rho_a(J(G_1)),\rho_a(J(G_2))\}$.
\end{proof}

\begin{corollary}
Let $G=G_1 \cup G_2$ be a clique-sum of $G_1$ and $G_2$. Then $J(G)$ is normal if and only if $J(G_1)$ and $J(G_2)$  are normal.
\end{corollary}

\begin{proof}
First assume that $J(G_1)$ and $J(G_2)$  are normal. For all $t \geq 1$,  $\overline{J(G)^t} \subset \overline{J(G_1)^t} \cap \overline{J(G_2)^t} =J(G_1)^t\cap J(G_2)^t=J(G)^t,$ by \Cref{clique-sum}. Therefore, $J(G)$ is normal.

Suppose $J(G)$ is normal. Then $J(G)_{P}\cap \kk[x : x \in V(G_1)]=J(G_1)$ is normal, where $P$ is the prime ideal $(x : x \in V(G_1)).$ Similarly, $J(G_2) $ is normal.   
\end{proof}
Cactus graphs are obtained by taking clique-sum of trees and cycles along vertices. We now compute their resurgence and asymptotic resurgence.

\begin{theorem}\label{cactus}
Let $G$ be a non-bipartite connected cactus graph. Then, $\displaystyle
\rho(J(G)) =\rho_a(J(G)) = \frac{n+1}{n},$ where $n$ is the number of vertices of a smallest  induced odd cycle in $G$. Moreover, $J(G)^{(nt+t)} \subset J(G)^{nt}$ for all $t \geq 1.$
\end{theorem}

\begin{proof}
Let $C$ be a smallest induced odd cycle in $G$. Then, by Theorem
\ref{odd-cycle}, $\displaystyle \rho_a(J(C)) =\frac{n+1}{n}$.
Now, by \Cref{induced-asym-resurgence}, $\displaystyle
\frac{n+1}{n}= \rho_a (J(C)) \leq \rho_a(J(G)) \leq \rho(J(G))
$. Hence it remains to prove that $\displaystyle \rho(J(G)) \leq \frac{n+1}{n}$ and $J(G)^{(nt+t)} \subset J(G)^{nt}$ for all $t \geq 1.$ We do this by induction on the number of blocks of $G$. Let
$b(G)$ denote the number of blocks in $G$. If $b(G)=1$, then
$G=C_n$, and hence, by \Cref{odd-cycle}, $\displaystyle
\rho(J(G)) \leq \frac{n+1}{n}$ and $J(G)^{(nt+t)} \subset J(G)^{nt}$ for all $t \geq 1.$ Next, assume that $b(G)>1$. Then,
$G$ has a cut vertex, say $v$.  Let $H_1,\ldots,H_k$ be the
connected components of $G \setminus v$. Now, let $G_1$ be the
induced subgraph of $G$ on the vertex set $V(H_1) \cup \{v\}$,
and $G_2$ be the induced subgraph of $G$ on the vertex set
$\{v\} \cup V(H_2) \cup \cdots \cup V(H_k)$. Note that $G=G_1
\cup G_2$  with $G_1\cap G_2 = K_1$. Therefore, $G$ is a
clique-sum of $G_1$ and $G_2$. Thus, by \Cref{clique-sum}, 
$\rho(J(G)) = \max\{\rho(J(G_1)),\rho(J(G_2))\}$. Since $G$ is non-bipartite, either $G_1$ or $G_2$ is non-bipartite. Without loss of generality, we may assume that the odd cycle $C$ is an induced subgraph of $G_1$. By induction, $\displaystyle\rho(J(G_1)) \leq \frac{n+1}{n}$ and $J(G_1)^{(nt+t)} \subset J(G_1)^{nt}$.
If $G_2$ is bipartite, then by \Cref{rho=1},
$\rho(J(G_2)) = 1$ and $J(G_2)^{(nt+t)} \subset J(G_i)^{nt}$ for all $t \geq 1$, and 
if $G_2$ is non-bipartite containing a smallest odd cycle of length $n_2$, then by induction
$\rho(J(G_2)) \leq \displaystyle{\frac{n_2+1}{n_2}}$ and $J(G_2)^{(n_2t+t)} \subset J(G_2)^{n_2t}$ for all $t \geq 1$.
Since $n \leq n_2$, $\displaystyle\rho(J(G))
=\max\{\rho(J(G_1)),\rho(J(G_2))\} \leq \frac{n+1}{n}$. Also, if $G_2$ is non-bipartite, then $J(G_2)^{(nt+t)} \subset J(G_2)^{nt}$ as  $\displaystyle \rho(J(G_2)) \leq  \frac{n_2+1}{n_2} \leq \frac{nt+t}{nt}$ for all $t \geq 1$. Therefore, by \Cref{clique-sum}, $J(G)^{(nt+t)}=J(G_1)^{(nt+t)}\cap J(G_2)^{(nt+t)}\subset J(G_1)^{nt}\cap J(G_2)^{nt}=J(G)^{nt}$ which completes the proof. 
\end{proof}

Given two graphs $G_1$ and $G_2$, another operation that produces a new graph is the join, $G_1*G_2$, of these two graphs. It would be interesting to find a connection between the resurgences of $G_1, G_2$ and $G_1*G_2$. Unlike in the case of clique-sum, we do not have a general answer here. However, if $\omega(G_i) = \chi(G_i)$ for $i=1, 2$, then one can compute resurgence and asymptotic resurgence. It can be seen that $\chi(G_1*G_2) = \chi(G_1) + \chi(G_2)$ and $\omega(G_1*G_2) = \omega(G_1) + \omega(G_2)$. Hence it follows from \Cref{res-upper2} that $\displaystyle \rho_a(J(G_1*G_2)) = \rho(J(G_1*G_2))= 2 - \frac{2}{\omega(G_1)+\omega(G_2)}$. Writing down $\omega(G_i)$ in terms of $\rho_a(J(G_i))$ or $\rho(J(G_i))$, we ask:

\begin{question}{\em
If $G_1$ and $G_2$ are non-trivial graphs, then is it true that 
\begin{enumerate}
    \item $\displaystyle \rho_a(J(G_1*G_2)) =2 - \frac{(2-\rho_a(J(G_1)))(2-\rho_a(J(G_2)))}{4-\rho_a(J(G_1))-\rho_a(J(G_2))}$?
    \item $\displaystyle \rho(J(G_1*G_2)) =2 - \frac{(2-\rho(J(G_1)))(2-\rho(J(G_2)))}{4-\rho(J(G_1))-\rho(J(G_2))}$?
\end{enumerate}
}
\end{question}
It was proved by Bocci and Harbourne, \cite{CB10}, that for a homogeneous ideal $I$, $\displaystyle  \frac{\alpha(I)}{\hat{\alpha}(I)} \leq \rho(I)$. It would be interesting to answer
\begin{question}
Classify graphs $G$ such that $\displaystyle \rho(J(G)) = \frac{\alpha(J(G))}{\hat{\alpha}(J(G))}$.
\end{question}
In \Cref{odd-cycle}, we proved that the odd cycles attain the lower bound. While we are unable to answer this question in general, we are able to classify graphs which are join of certain bipartite graphs.

\begin{proposition}\label{join-res}

\begin{enumerate}
\item If $G =K_m^c*H$, where $H$ is a non-trivial bipartite graph on $n$ vertices, then $\displaystyle \rho(J(G))=\frac{\alpha(J(G))}{\hat{\alpha}(J(G))}$ if and only if $\displaystyle m=\alpha(J(H))= \frac{n}{2}.$
\item If $G = K_{n_1, \ldots, n_k}$, then $\displaystyle\rho(J(G)) = \frac{\alpha(J(G))}{\hat{\alpha}(J(G))}$ if and only if $n_1=\cdots=n_k.$
\end{enumerate}
\end{proposition}
\begin{proof} 

(1) First we describe $J(G)$ and $J(G)^{(2)}$. Write $V(K_m^c) = \{x_1,\ldots,x_m\}$ and $V(H) = \{y_1,\ldots,y_n\}$. 
Let $\Gamma$ be a minimal vertex cover of $G$. If $x_i \notin
\Gamma$ for some $i$, then $\{y_1,\ldots,y_n\}\subset \Gamma$,
since $\{x_i,y_j\} \in E(G)$ for every $1 \leq j \leq n$.
Therefore, $\Gamma=\{y_1,\ldots,y_n\}$. Next, we assume that
$\{x_1, \ldots, x_m\} \subset \Gamma$. Note that
$\Gamma\setminus \{x_{1},\ldots,x_m\}$ is a minimal vertex cover
of $H$. Therefore, 
\begin{equation}\label{eqn-jg}
    J(G)=(x_1\cdots x_m)J(H)+(y_1\cdots y_n).
\end{equation}

 Note that the only odd cycles in $G$ are the triangles on the
vertices $\{x_i, y_j, y_k\}$, where $1 \leq i \leq m$  and
$\{y_j, y_k\}\in  E(H)$. Let $C$ be such a cycle on the vertex set $\{x_i,y_j,y_k\}$. Then, every vertex of $H$ is adjacent to $x_i$ and every vertex of $K_m^c$ is adjacent to $y_j$ (and $y_k$). Therefore, every vertex of $G$ is adjacent to some vertex of $C$.  Thus, by \cite[Proposition 5.3]{HHTV}, 
\begin{equation}\label{eqn-jg2}
 J(G)^{(2)}=J(G)^2+ \Big(  \prod_{x \in V(G)} x\Big).
\end{equation}
It follows from \eqref{eqn-jg} and \eqref{eqn-jg2} that
\begin{eqnarray}\label{alpha}
 \alpha(J(G)) &= & \min \{n, \alpha(J(H))+m\} \hspace*{1cm} \text{ and }\\
\alpha(J(G)^{(2)})& = & \min \{n+m,
\alpha(J(G)^2)\}=\min\{2\alpha(J(H))+2m, 2n, n+m\}. \nonumber
\end{eqnarray}
Now, by \Cref{walds},
\begin{eqnarray}\label{alphahat}
\hat{\alpha}(J(G))& = &\min\Big\{\alpha(J(H))+m, n,
\frac{n+m}{2}\Big\}.
\end{eqnarray}
It may also be observed that by \Cref{res-upper2}, $\rho(J(G)) \displaystyle = \frac{4}{3}$.

We now prove that $\displaystyle \frac{\alpha(J(G))}{\hat{\alpha}(J(G))} = \frac{4}{3}$ if and only if $\displaystyle m =  \alpha(J(H)) = \frac{n}{2}.$
Assume that $\displaystyle m=\alpha(J(H))=\frac{n}{2}$. Therefore, by \eqref{alpha}, $\alpha(J(G))=n$ and by \eqref{alphahat}, $\displaystyle \hat{\alpha}(J(G))= \frac{3n}{4}.$ Hence $\displaystyle \frac{\alpha(J(G))}{\hat{\alpha}(J(G))} =\frac{4}{3}.$ Conversely, we assume that $\displaystyle  \frac{\alpha(J(G))}{\hat{\alpha}(J(G))} = \frac{4}{3}.$ Note that as $H$ is a bipartite graph, $\displaystyle \alpha(J(H)) \leq \frac{n}{2}$.  If $m+\alpha(J(H)) <n,$ then $\alpha(J(G))= m+\alpha(J(H))$ and $\displaystyle \hat{\alpha}(J(G))=\min\Big\{\alpha(J(H))+m,  \frac{n+m}{2}\Big\}$. Since $\displaystyle  \frac{\alpha(J(G))}{\hat{\alpha}(J(G))} =\frac{4}{3},$ we must have $\displaystyle \hat{\alpha}(J(G))=  \frac{n+m}{2}$. Thus $\displaystyle  \frac{\alpha(J(G))}{\hat{\alpha}(J(G))}= \frac{2m+2\alpha(J(H))}{n+m} =\frac{4}{3}$, and hence, $m+3 \alpha(J(H))=2n.$ This contradicts our assumption that $m+ \alpha(J(H))<n.$ Thus, if $m+ \alpha(J(H))<n$, then $\displaystyle  \frac{\alpha(J(G))}{\hat{\alpha}(J(G))}\neq \frac{4}{3}$. Next assume that $m+\alpha(J(H))>n.$ If $m \geq n$, then $\alpha(J(G))=n=\hat{\alpha}(J(G))$ which contradicts the assumption that $\displaystyle  \frac{\alpha(J(G))}{\hat{\alpha}(J(G))} =\frac{4}{3}$. So assume that $m<n.$ Thus, $\alpha(J(G))=n$ and $\displaystyle\hat{\alpha}(J(G))=\frac{n+m}{2}.$ Since $\displaystyle  \frac{\alpha(J(G))}{\hat{\alpha}(J(G))} =\frac{2n}{n+m} =\frac{4}{3},$ we have $n=2m.$ Therefore, $\displaystyle m < \alpha(J(H)) \leq \frac{n}{2}=m,$ a contradiction. Thus, if $m+ \alpha(J(H))>n$, then $\displaystyle  \frac{\alpha(J(G))}{\hat{\alpha}(J(G))}\neq \frac{4}{3}$. Now, we assume that $m+\alpha(J(H))=n$. Note that, in this case, $m < n$.  Hence $\alpha(J(G))=n$ and  $\displaystyle\hat{\alpha}(J(G))=\frac{n+m}{2}$ so that $\displaystyle  \frac{\alpha(J(G))}{\hat{\alpha}(J(G))} =\frac{2n}{n+m} =\frac{4}{3}.$ Therefore $m+\alpha(J(H))=n=2m.$ 

\par (2)  Let $V(G) = V_1 \sqcup \cdots \sqcup V_k$ with $|V_i| = n_i$. Without loss  of generality, we may assume that $n_1 \leq \cdots \leq  n_k$. 
First we compute $\alpha(J(G))$ and $\hat{\alpha}(J(G))$.
 By \cite[Proposition 5.3]{HHTV}, $\displaystyle J(G)^{(2)}=J(G)^2+ \Big(  \prod_{x \in V(G)} x\Big)$. Therefore $\alpha(J(G)^{(2)})=\min \{n, \alpha(J(G)^2)\}=\min \{n, 2  \alpha(J(G))\}$.   The minimal vertex covers of $G$ are $V(G) \setminus V_i$ for $1 \leq  i \leq k$. Thus, every vertex cover of $G$ has at least $n-n_k$ elements, and hence,  $\alpha(J(G)) = n-n_k$ and   $\alpha(J(G)^{(2)})=\min\{n, 2n-2n_k\}$.
By \Cref{walds}, $\displaystyle \hat{\alpha}(J(G))=\frac{\alpha(J(G)^{(2)})}{2}=\min\left\{\frac{n}{2}, n-n_k\right\}.$  Also, from \Cref{res-upper2},  we get $ \displaystyle \rho(J(G))=2 - \frac{2}{k}$.
 
Now, assume that $n_1=\cdots=n_k$. Therefore, $\displaystyle\alpha(J(G))=(k-1)n_k$ and $\displaystyle\hat{\alpha}(J(G))=\frac{kn_k}{2}.$ Thus, $\displaystyle\frac{\alpha(J(G))}{\hat{\alpha}(J(G))}=\rho(J(G)).$ Conversely, suppose that $\displaystyle\rho(J(G))=\frac{\alpha(J(G))}{\hat{\alpha}(J(G))}.$ Then, $\displaystyle k(n-n_k)=2(k-1) \min \Big\{ \frac{n}{2}, n-n_k\Big\}$. For this equality to hold, for $k \geq 3$, we must have $\displaystyle \min \Big\{ \frac{n}{2}, n-n_k\Big\}= \frac{n}{2}$. Therefore,  $\displaystyle k(n-n_k)=2(k-1) \frac{n}{2} $. Thus, $n=kn_k$  which further implies that $n_1=\cdots=n_k.$
\end{proof}

\section{resurgence of edge ideals}
In this section, we study the resurgence and asymptotic resurgence of edge ideals of graphs. We begin by studying the relationship between the resurgence and asymptotic resurgence of edge ideals of a graph and its induced subgraphs. 

\begin{lemma}\label{tech1}
Let $G$ be a graph, and $H$ be a non-trivial induced subgraph of $G$. Then, $I(H)^{(s)}=I(G)^{(s)} \cap R_H$, and $I(H)^s=I(G)^s\cap R_H$, where $R_H =\kk[x_i : \; x_i \in V(H)]$ for every $s \geq 1$.
\end{lemma}
\begin{proof}
Let $\mathcal{C}(G)$ be the set of vertex covers of $G$. For $C\in \mathcal{C}(G)$, let $P_C$ denote the ideal generated by $\{x_i : x_i \in C\}$. Then, \begin{align*} \displaystyle 
 I(G)^{(s)} \cap R_H &= \big(\bigcap_{C \in \mathcal{C}(G)} P_C^s\big) \cap R_H =\bigcap_{C \in \mathcal{C}(G)}\big(P_C^s \cap R_H \big)\\ &= \bigcap_{C \in \mathcal{C}(G)}\big(P_{C \cap V(H)}^s \big)=\bigcap_{C \in \mathcal{C}(H)}P_{C}^s=I(H)^{(s)}.
\end{align*}
Clearly, $I(H)^s \subset I(G)^s \cap R_H$. Let $u \in I(G)^s\cap R_H$ be a monomial. Then, there exist $e_{1},\ldots,e_{s} \in E(G)$ such that $(e_{1}\cdots e_{s}) \mid u$. Since $u \in R_H$, supp$(u) \subset V(H)$, and hence, supp$(e_{1}\cdots e_{s}) \subset \text{supp}(u) \subset V(H)$. Consequently,  $e_{1},\ldots,e_{s} \in E(H)$. Thus, $u \in I(H)^s$. Hence, $I(H)^s =I(G)^s \cap R_H$. 
\end{proof}

\begin{proposition}\label{tech2}
Let $G$ be a graph, and $H$ be a non-trivial  induced subgraph of $G$. Then
\begin{enumerate}
    \item $\rho(I(H)) \leq \rho(I(G))$.
    \item $\rho_a(I(H)) \leq \rho_a(I(G))$.
\end{enumerate}  
\end{proposition}
\begin{proof} \par (1) Let $s,t  \in \mathbb{N}$ be such that $I(H)^{(s)} \not\subset I(H)^t$. Then, it follows from \Cref{tech1} that $I(G)^{(s)} \not\subset I(G)^t$. Consequently, $\displaystyle \rho(I(G)) \geq \frac{s}{t}$, and hence, $\rho(I(G))$ is an upper bound for the set $\displaystyle \left\{ \frac{s}{t} \; : \; s,t \in \mathbb{N}, \text{ and } I(H)^{(s)} \not\subset I(H)^t\right\}$. Thus, $\rho(I(G)) \geq \rho(I(H))$. 

\par (2) Let $s,t  \in \mathbb{N}$ be  such that $I(H)^{(sr)} \not\subset I(H)^{tr}$ for $r\gg 0$. Then, it follows from \Cref{tech1} that $I(G)^{(sr)} \not\subset I(G)^{tr}$ for $r \gg 0$. Consequently, $\displaystyle \rho_a(I(G)) \geq \frac{s}{t}$, and hence, $\rho_a(I(G))$ is an upper bound for the set $\displaystyle \left\{ \frac{s}{t} \; : \; s,t \in \mathbb{N}, \text{ and }  I(H)^{(sr)} \not\subset I(H)^{tr}, \text{ for all } r \gg 0\right\}$. Thus, $\rho_a(I(G)) \geq \rho_a(I(H))$. 
\end{proof}

It was shown by Simis, Vasconcelos and Villarreal in \cite{SVV} that $G$ is bipartite if and only if $I(G)^{(s)} = I(G)^s$ for all $s \geq 1$. For a homogeneous ideal $I$, $\rho(I) = 1$ or $\rho_a(I) = 1$ need not necessarily imply that $I^{(s)} =  I^s$ for all $s \geq1$ while the converse always hold. We prove that these implications do hold for edge ideals of graphs.
\begin{theorem}\label{rhoI=1}
Let $G$ be a graph.  Then, 
\[
\rho(I(G)) =1 \text{ if and only if } G \text{ is a bipartite graph if and only if } \rho_a(I(G)) =1.
\]
\end{theorem}
\begin{proof}
Assume that $G$ is a bipartite graph. Then, by \cite[Theorem 5.9]{SVV}, $I(G)^{(s)}=I(G)^s$ for all $s \geq 1$. Therefore, $\rho(I(G))= \rho_a(I(G))=1$. Now, assume that $G$ is not a bipartite graph. Consequently, $G$ has an induced odd cycle, say $C_{2n+1}$. By \Cref{tech2}, $\displaystyle \rho(I(G) )\geq \rho_a(I(G)) \geq \rho_a(I(C_{2n+1}))=\frac{2n+2}{2n+1} >1,$ where the last equality follows from \cite[Theorem 5.11]{JKV19}. Hence, the assertion follows.
\end{proof}

Thanks to \Cref{res-sum1} and \Cref{rhoI=1}, in the  study of the resurgence of edge ideals of graphs, we need to consider only those graphs whose each component contains an odd cycle. For the rest of the section, we set the following notation.
\begin{notation}\label{notn}{\em
Let $G$ be a graph obtained by taking clique-sum of bipartite graphs and odd cycles. Let $n_1,\ldots,n_r$ be positive integers such that $n_1< \ldots < n_r$, $C_{2n_i+1}$ is an induced cycle in $G$ for all $1 \leq i \leq r$, and if $C$ is an induced odd cycle in $G$, then $C\simeq C_{2n_i+1}$ for some $i$. For an odd cycle $C$, set $\displaystyle u_C= \prod_{x \in V(C)} x$. Let $J_{n}(G)$ be the ideal generated by $\{u_C \; : \; C \text{ is an induced odd cycle of } $G$ \text{ length  } 2n+1 \}$.
}
\end{notation}
Gu  et al. obtained a nice decomposition for the symbolic powers of edge ideals of unicyclic graphs in terms of ordinary powers, \cite{GHOS}. Using the property that clique-sum of implosive graphs is an implosive graph, we obtain a similar decomposition for a more general class, namely clique-sum of odd cycles and bipartite graphs (see \cite{FGR} for the definition of implosive graphs and its properties). By \cite[Theorem 2, Corollary 1b]{FM}, any indecomposable induced subgraph of G is either an odd cycle or an edge (see \cite{FM} for the definition and properties of indecomposable graphs). Thus,  by \cite[Lemma 2.1]{MRV}, $$\R_s(I(G)) = R[I(G)t, J_{n_1}(G)t^{n_1+1}, \ldots, J_{n_r}(G)t^{n_r+1}   ].$$ By comparing the graded components of degree $s$, we get
\begin{theorem}\label{gen}
Let $G$ be a graph as given in \Cref{notn}. Then for all $s \geq 2$, $$I(G)^{(s)} = \sum_{\substack{t \geq 0, \;  a_i\geq 0,\\ s=t+(n_1+1)a_1+ \cdots + (n_r+1)a_r}} I(G)^{t}J_{n_1}(G)^{a_1} \cdots J_{n_r}(G)^{a_r}.$$
\end{theorem}

As an immediate consequence, we obtain Waldschmidt constant and asymptotic resurgence for this class of graphs.
\begin{lemma}\label{res-cactus-edge}
Let $G$ be a graph as in \Cref{notn}. Then :  \begin{enumerate}
	\item For all $s \geq 1$, $\displaystyle \alpha(I(G)^{(s)}) = 2s-\floor[\Big]{\frac{s}{n_1+1}}.$
	\item $\displaystyle \hat{\alpha}(I(G)) = \frac{2n_1+1}{n_1+1}$.
	\item $\displaystyle \rho_a(I(G))=\frac{2n_1+2}{2n_1+1}.$
\end{enumerate}
\end{lemma}

\begin{proof} \par 
(1) If $s \leq n_1$,  then by \Cref{gen}  $I(G)^{(s)} =I(G)^s$. Therefore, $\displaystyle \alpha(I(G)^{(s)}) =\alpha(I(G)^s)=2s$. Assume that $s \geq n_1+1$. Then
\begin{eqnarray*}
	\alpha(I(G)^{(s)}) &= & \min \{\alpha(I(G)^{t}J_{n_1}(G)^{a_1} \cdots J_{n_r}(G)^{a_r})\; : \; t \geq 0, \; a_i \geq 0, \text{ and } \\ && \;\;\;\;\;\;\;\;\;\;\;\;\;\;\;\;\;\;\;\;\;\;\; s=t+(n_1+1)a_1+ \cdots + (n_r+1)a_r \}\\ & = &  \min \{ 2t+ (2n_1+1){a_1}+ \cdots (2n_r+1){a_r} \; : \; t \geq 0, \; a_i \geq 0, \text { and }\\ && \;\;\;\;\;\;\;\;\;\;\;\;\;\;\;\;\;\;\;\;\;\;\; s=t+(n_1+1)a_1+ \cdots + (n_r+1)a_r \} \\ & = &\min \{ 2s-a_1- \cdots -a_r \; : \; t \geq 0, \; a_i \geq 0, \text { and }\\ && \;\;\;\;\;\;\;\;\;\;\;\;\;\;\;\;\;\;\;\;\;\;\;\; s=t+(n_1+1)a_1+ \cdots + (n_r+1)a_r \}.
\end{eqnarray*} 
For each $i$, there exist non-negative integers $q_i, r_i $ such that $s=q_i(n_i+1)+r_i$ and $r_i \leq n_i$.
If $a_1+\cdots +a_r \leq q_1$, then $\displaystyle 2s-a_1-\cdots -a_r \geq 2s-q_1=  2s-\floor[\Big]{\frac{s}{n_1+1}}$.
Now assume that $a_1 + \cdots + a_r > q_1$. 
Let  $a_1+\cdots +a_r = q_1+p$, for some $p >0$.  Then,
 \begin{align*}
 s&=t+(n_1+1)a_1+\cdots+(n_r+1)a_r\\&=t+ (n_1+1)(a_1+\cdots+a_r)+\sum_{j=2}^r(n_j-n_1)a_j\\&=t+(q_1+p)(n_1+1)+\sum_{j=2}^r(n_j-n_1)a_j\\&= s+t+\sum_{j=2}^r(n_j-n_1)a_j+p(n_1+1) -r_1>s, 
 \end{align*}
 which is a contradiction. Therefore, $a_1+\cdots +a_r \leq q_1$ which implies that $\displaystyle \alpha(I(G)^{(s)}) \geq 2s-q_1=  2s-\floor[\Big]{\frac{s}{n_1+1}}$. 
If we take $a_1=q_1$, then $a_2=\cdots=a_r=0$ and we get a minimal
generator in degree $2s-q_1$. 
Hence, $\displaystyle \alpha(I(G)^{(s)}) =2s-q_1 =2s-\floor[\Big]{\frac{s}{n_1+1}} $. 
 
\par (2) It follows from (1) that 
$$\displaystyle \hat{\alpha}(I(G)) = \lim_{s \to \infty}\frac{\alpha(I(G)^{(s)})}{s}=\lim_{s \to \infty}\frac{2s-\floor[\big]{\frac{s}{n_1+1}}}{s}=\frac{2n_1+1}{n_1+1}.$$

\par (3) It follows from (2) and  \cite[Theorem 3.12]{MCM19} that $\displaystyle \rho_a(I(G))=\frac{2}{\hat{\alpha}(I(G))}=\frac{2n_1+2}{2n_1+1}.$
\end{proof}

From \cite[Theorem 4.6]{CSE}, it is known that $\displaystyle \hat{\alpha}(I(G)) = \frac{\chi_f(G)}{\chi_f(G)-1},$ where $\chi_f(G)$ denote the fractional chromatic number of $G$. Explicit formula for fractional chromatic number is known only for a very few classes of  graphs. \Cref{res-cactus-edge} $(2)$ allows us to compute $\chi_f(G)$ for this class of graphs.

\begin{corollary}
If $G$ is a graph as in \Cref{notn}, then $\displaystyle \chi_f(G) = 2 + \frac{1}{n_1}.$
\end{corollary}

We now proceed to study the resurgence of graphs described in \Cref{notn}. We compute the resurgence of edge ideals of graphs having cycles of equal length, i.e., $r=1$ as per \Cref{notn}. 
For a graph $G$ and $u, v \in V(G),$ the distance between $u$ and $v$, denoted by $d(u,v)$ is the length of a shortest path in $G$ from $u$ to $v$. For two subgraphs  $H$ and $ H'$ of $G$, the distance between $H$ and $H'$, denoted by $d(H, H')$ is the minimum of the set $\{d(u,v) \mid u \in V(H) \text{ and } v\in V(H')\}$.

\begin{lemma}\label{tech3}
Let $G$ be a clique-sum of bipartite graphs and cycles of size $2n+1$. Let $k_n(G)$ denote the maximum cardinality of a collection of induced odd cycles in $G$ such that for any two cycles $C$ and $C'$ in the collection, $d(C, C') \geq 2$.  Then, 
for every $b > k_n(G)$,  $$\displaystyle J_n(G)^b \subset I(G)^{bn+\ceil[\big]{\frac{b-k_n(G)}{2}}}.$$ 
\end{lemma}
\begin{proof}
Let $C_1, \ldots, C_l$ be cycles of size $2n+1$ in $G$. For each  $1 \leq i \leq l$, set $\displaystyle u_i= \prod_{x \in V(C_i)} x$. Then, $\{u_{1},\ldots, u_{l}\}$ is the minimal generating
set of $J_n(G)$.
We prove the lemma by induction on $b >k_n(G)$. Assume that $b=k_n(G)+1$. Let   $\displaystyle u \in J_n(G)^b $ be such that $\displaystyle u=u_{1}^{b_{1}} \cdots u_l^{b_l}$ with $b_i \geq 0$ for $1 \leq i \leq l$ and $\displaystyle \sum_{i=1}^l b_i =b$. We define $j_u^b := |\{ i : b_i  \text{ is odd}\}|$. If  $0 \leq j_u^b \leq k_n(G)$, then $\displaystyle u =u_{1}^{b_{1}} \cdots u_l^{b_l} \in I(G) ^{bn+{\frac{b-j_u^b}{2}}} $ as $\displaystyle u_i^{b_i} \in I(G)^{b_in+\floor[\big]{\frac{b_i}{2}}}$ for all $i$. Since $ \displaystyle \frac{b-j_u^b}{2} \geq \ceil[\Big]{\frac{b-k_n(G)}{2}}$, we have $\displaystyle u \in I(G)^{bn+\ceil[\big]{\frac{b-k_n(G)}{2}}} $. If $j_u^b =k_n(G)+1$, then each nonzero $b_i$'s are one. There exist $ 1 \leq i < j \leq l$ such that  the exponents  of  $u_i$ and $u_j$  in $ u \in J_n(G)^b$ are one and $d(C_i,C_j) \leq 1$. Since $d(C_i,C_j) \leq 1$, $u_iu_j \in I(G)^{2n+1}$. Set $\displaystyle u' =\frac{u}{u_iu_j}$. Then   $u' \in I(G)^{(b-2)n}$ as $u' \in J_n(G)^{b-2}$ and $\displaystyle u_i \in I(G)^{n}$ for all $i$. Thus, $\displaystyle u \in I(G)^{bn+1}= I(G)^{bn+\ceil[\big]{\frac{b-k_n(G)}{2}}} $ as $b=k_n(G)+1.$ Hence, the base case is true. 
		
Assume that $ b > k_n(G)+1$.  Let   $u \in J_n(G)^b $ be such that  $u=u_{1}^{b_{1}} \cdots u_l^{b_l}$ with $b_i$'s are non-negative
integers and  $\displaystyle \sum_{i=1}^l b_i =b$.  If  $0 \leq j_u^b \leq k_n(G)$, then $\displaystyle u =u_{1}^{b_{1}} \cdots u_l^{b_l} \in I(G)^{bn+{\frac{b-j_u^b}{2}}} $ as $\displaystyle u_i^{b_i} \in I(G)^{b_in+\floor[\big]{\frac{b_i}{2}}}$ for all $i$. Since $\displaystyle  \frac{b-j_u^b}{2} \geq \ceil[\Big]{\frac{b-k_n(G)}{2}}$, we have $u \in I(G)^{bn+\ceil[\big]{\frac{b-k_n(G)}{2}}} $. If $j_u^b \geq k_n(G)+1$, then there exists $ 1 \leq i < j \leq l$ such that  the exponents  of  $u_i$ and $u_j$  in $ u \in J_n(G)^b$ are odd and $d(C_i,C_j) \leq 1$. Since $d(C_i,C_j) \leq 1$, $u_iu_j \in I(G)^{2n+1}$. Set $\displaystyle u' =\frac{u}{u_iu_j}$. Then $u' \in J_n(G)^{b-2}$. Suppose $b-2> k_n(G)$, then by induction, we have $\displaystyle u' \in I(G)^{(b-2)n+\ceil[\big]{\frac{b-2-k_n(G)}{2}}} $.  Therefore, $\displaystyle u \in  I(G)^{(b-2)n+\ceil[\big]{\frac{b-2-k_n(G)}{2}}+2n+1} =I(G)^{bn+\ceil[\big]{\frac{b-k_n(G)}{2}}}  $. If $b-2 = k_n(G)$, then $u' \in I(G)^{(b-2)n}$, and hence, $\displaystyle u \in I(G)^{bn+1}= I(G)^{bn+\ceil[\big]{\frac{b-k_n(G)}{2}}} $. Thus, in either case, we have $u \in I(G)^{bn+\ceil[\big]{\frac{b-k_n(G)}{2}}}$. Hence, the assertion follows.  
\end{proof} 

We now compute the resurgence of graphs obtained by taking clique-sum of bipartite graphs and several odd cycles of equal length.

\begin{theorem}\label{res-edgeideal}
Let $G$ be a clique-sum of bipartite graphs and cycles of size $2n+1$. Let $k=k_n(G)$ be as in \Cref{tech3}. Then,
\[ \rho(I(G)) =  \left\{
\begin{array}{ll}
 \displaystyle \frac{2n+2}{2n+1} & \text{ if } k = 1, \\
 & ~ \\
\displaystyle \frac{kn+k}{kn+1} & \text{ if } k \geq 2.
\end{array}\right.
\]
\end{theorem}
\begin{proof}
First, we assume that $k \geq 2$. Let $s, t$ be positive integers such that $I(G)^{(s)} \not\subset I(G)^t$. If $ s \leq n$, then by Theorem \ref{gen}, $I(G)^{(s)} =I(G)^s$. Therefore, $\displaystyle \frac{s}{t} \leq \frac{s}{s+1} \leq \frac{kn+k}{kn+1}$. Assume that $s \geq n+1$. Then there exists non-negative integers $q_1,r_1$ such that $s=q_1(n+1)+r_1$ with $r_1 \leq n$. By Theorem \ref{gen}, $\displaystyle I(G)^{(s)} =\sum_{i=0}^{q_1}I(G)^{s-i(n+1)}J_{n}(G)^i.$ 
Following \Cref{notn}, it can be seen that $J_n(G)^i \subseteq I(G)^{ni}$ for all $i \geq 1$. Thus we have $I(G)^{s-i(n+1)}J_{n}(G)^i \subset I(G)^{s-i}$ for $1 \leq i \leq q_1$ so that $I(G)^{(s)} \subset I(G)^{s-q_1}$. Hence, for $q_1 \leq k$, $\displaystyle \frac{s}{t} \leq \frac{s}{s-q_1+1} =\frac{q_1n+q_1+r_1}{q_1n+r_1+1}\leq \frac{kn+k}{kn+1}$. If $q_1 >k$, then   by Lemma \ref{tech3}, we have $\displaystyle I(G)^{s-i(n+1)}J_{n}(G)^i \subset I(G)^{s-i+\ceil[\big]{\frac{i-k}{2}}}$,  for $k < i \leq q_1$. Thus, $\displaystyle I(G)^{(s)} \subset I(G)^{s-q_1+\ceil[\big]{\frac{q_1-k}{2}}}$, if $q_1 >k$. Therefore, $\displaystyle \frac{s}{t} \leq \frac{s}{s-q_1+\ceil[\big]{\frac{q_1-k}{2}}+1} \leq \frac{s}{s-q_1+{\frac{q_1-k}{2}}+1} = \frac{2s}{2s-q_1-k+2} \leq \frac{kn+k}{kn+1}$. Hence, we have $\displaystyle \rho(I(G))\leq \frac{kn+k}{kn+1}$. Let $C_1, \ldots, C_k$ be cycles of size $2n+1$ with $d(C_i, C_j) \geq 2$ for $i \neq j$. Then $u_{C_1}\cdots u_{C_k} \in J_n(G)^k \subset I(G)^{kn}$. Note that $u_{C_i}$ is divisible by a product of at most $n$-edges. For $u_{C_1}\cdots u_{C_k}$ to be in $I(G)^{kn+1}$, there must exist $x_i\in V(C_i), x_j \in V(C_j), i \neq j$ and $\{x_i, x_j\} \in E(G)$. Since $d(C_i, C_j) \geq2,$ this is not possible. Hence $J_n(G)^k \not\subset I(G)^{nk+1}$.
Thus, $\displaystyle I(G)^{(kn+k)} = \sum_{i=1}^k I(G)^{s-i(n+1)}J_{n}(G)^i \not\subset I(G)^{kn+1}$. Therefore, $\displaystyle \rho(I(G)) \geq \frac{kn+k}{kn+1}$. Hence, $\displaystyle \rho(I(G) )=\frac{kn+k}{kn+1}$.

Now, assume that $k=1$. Let $s, t$ be positive integers such that $I(G)^{(s)} \not\subset I(G)^t$. If $ s \leq n$, then by Theorem \ref{gen}, $I(G)^{(s)} =I(G)^s$. Therefore, $\displaystyle \frac{s}{t} \leq \frac{s}{s+1} \leq \frac{2n+2}{2n+1}$. Assume that $s \geq n+1$. Then there exists non-negative integers $q_1,r_1$ such that $s=q_1(n+1)+r_1$ with $r_1 \leq n$. By Theorem \ref{gen}, $\displaystyle I(G)^{(s)} =\sum_{i=0}^{q_1}I(G)^{s-i(n+1)}J_{n}(G)^i.$ Then, by \Cref{tech3}, we have $\displaystyle I(G)^{s-i(n+1)}J_{n}(G)^i \subset I(G)^{s-i+\ceil[\big]{\frac{i-1}{2}}}=I(G)^{s-\ceil[\big]{\frac{i}{2}}}$  for $1 \leq i \leq q_1$. Thus, $\displaystyle I(G)^{(s)} \subset I(G)^{s-\ceil[\big]{\frac{q_1}{2}}}$. Therefore, $\displaystyle \frac{s}{t} \leq \frac{s}{s-\ceil[\big]{\frac{q_1}{2}}+1} \leq \frac{s}{s-{\frac{q_1+1}{2}}+1}= \frac{2s}{2s-q_1+1} \leq \frac{2n+2}{2n+1}$. Hence, we have $\displaystyle \rho(I(G))\leq \frac{2n+2}{2n+1}$. Now, by \Cref{res-cactus-edge}, $\displaystyle \rho(J(G)) \geq \rho_a(J(G)) =\frac{2n+2}{2n+1}$, and hence,  $\displaystyle\rho(I(G) )=\frac{2n+2}{2n+1}$.
\end{proof}

\begin{remark}{\em
In \cite[Question 3.16]{MCM19}, the authors ask whether the asymptotic resurgence and the resurgence are equal for the class of edge ideals. This questions is known to have a negative answer, for example, see \cite[Example 4.4]{DD20}.
Using \Cref{res-edgeideal}, we give a class of examples for which the answer is negative. For, if $G$ is a graph obtained by taking clique-sum of a bipartite graph and $k$ odd cycles of size $2n+1$ with $k > 2$, then $\displaystyle \rho_a(I(G)) = \frac{2n+2}{2n+1} < \frac{kn+k}{kn+1} = \rho(I(G)).$
}
\end{remark}

\begin{remark}{\em
As of now, the only known upper bound for the resurgence of edge ideals is two. It may be noted that the upper bound given in \Cref{res-upper1} does not work for the case of edge ideals. For example, let $G$ be the graph as in \Cref{res-edgeideal} with $n=1$ and $k = 3$. Then $\rho(I(G)) = \frac{3}{2}.$ The chromatic number of $G$ is  equal to the maximum of the chromatic numbers of a triangle and a bipartite graph. Hence $\chi(G) = 3$. Therefore, $2 - \frac{2}{\chi(G)} = \frac{4}{3} < \frac{3}{2} = \rho(I(G))$.
}
\end{remark}

It may be noted that the primary decomposition of the edge ideals are much more complex in nature, compared the primary decomposition of the cover ideals. This is possibly one reason why the study of symbolic powers of edge ideals are more challenging than that of the cover ideals. It would be interesting to compute sharp upper and lower bounds, similar to the ones in \Cref{res-upper1}, for the class of edge ideals.

\noindent
\textbf{Acknowledgement:} We would like to thank Huy T\`ai H\`a and Michael DiPasquale for their comments on the initial draft of this paper. We also wish to thank the reviewer for the valueable comments which made the exposition better. The second author is funded by SERB National Post-Doctoral Fellowship, grant number PDF/2020/001436.
	\bibliographystyle{plain}  
	\bibliography{bilbo}

\end{document}